%% file: EIPpC.tex
\documentclass{amsart}

\usepackage[T1]{fontenc}
\usepackage[utf8]{inputenc}
\usepackage{stmaryrd}
\usepackage{MnSymbol}
\usepackage[english]{babel}
\usepackage[english=american,french=guillemets]{csquotes}
\MakeOuterQuote{"}
\usepackage[all]{xy}
\usepackage{enumitem}
\setlist{nosep}
\usepackage[simpl]{mythm}
\usepackage[simpl]{mybasicmacro}
\usepackage[colorlinks=true, breaklinks=true, urlcolor= black, linkcolor= black, citecolor= black, bookmarksopen=true,linktocpage=true,plainpages=false,pdfpagelabels]{hyperref}

\author{Samaria Montenegro \and Silvain Rideau}
\thanks{Partially supported by ValCoMo (ANR-13-BS01-0006)}
\title{Imaginaries, invariant types and pseudo $p$-adically closed fields}
\date{\today}


\newcommand{\PpC}[1][]{\ifstrempty{#1}{}{#1{-}}\mathrm{P}p\mathrm{C}}

\newcommand{\pCFG}{p\mathrm{CF}^{\Geom}}
\newcommand{\pCF}{p\mathrm{CF}}
\newcommand{\trdeg}[1]{\mathrm{trdeg}(#1)}
\newcommand{\Balls}{\mathbf{B}}

\newcommand{\LP}{\cL_{\mathrm{Mac}}}
\newcommand{\Geomim}{\Geom^{\mathrm{im}}}

\newcommand{\qfindep}{\indep^{\mathrm{i},\mathrm{qf}}}
\newcommand{\Pow}[2][]{\mathrm{P}^{#1}_{#2}}
\newcommand{\pM}[2][]{\ifstrempty{#1}{M}{#1}_{#2}}
\newcommand{\aM}[2][]{\alg{\ifstrempty{#1}{M}{#1}_{#2}}}
\newcommand{\pcl}[1]{\overline{#1}^{\mathrm{p}}}
\newcommand{\equivL}[1][]{\equiv_{\cL\ifstrempty{#1}{}{(#1)}}}
\newcommand{\equivp}[2][]{\equiv_{\cL_{\ifstrempty{#1}{i}{#1}}(#2)}}
\newcommand{\equiva}[2][]{\equiv_{\bcL_{\ifstrempty{#1}{i}{#1}}(#2)}}
\renewcommand{\Latt}[2][]{\LattN^{#1}_{#2}}
\renewcommand{\val}{\mathrm{v}}
\newcommand{\bcL}{\alg{\mathcal{L}}}
\newcommand{\acla}[1][]{\acl_{\bcL_{\ifstrempty{#1}{i}{#1}}}}
\newcommand{\aclp}[1][]{\acl_{\cL_{\ifstrempty{#1}{i}{#1}}}}
\newcommand{\dcla}[1][]{\dcl_{\bcL_{\ifstrempty{#1}{i}{#1}}}}
\newcommand{\dclp}[1][]{\dcl_{\cL_{\ifstrempty{#1}{i}{#1}}}}
\newcommand{\tpa}[1][]{\tp_{\bcL_{\ifstrempty{#1}{i}{#1}}}}
\newcommand{\tpp}[1][]{\tp_{\cL_{\ifstrempty{#1}{i}{#1}}}}
\newcommand{\qfLG}{\bcL}
\newcommand{\qfequiv}[1][]{\equiv^{\mathrm{qf}}_{#1}}
\newcommand{\dpr}{\mathrm{dp}}

\usepackage{color}
\definecolor{orange}{rgb}{1,0.5,0}

\definecolor{secret}{rgb}{0,0.5,1}

\begin{document}

\begin{abstract}
In this paper, we give a general criterion for elimination of imaginaries using an abstract independence relation. We also study germs of definable functions at certain well-behaved invariant types. Finally we apply these results to prove the elimination of imaginaries in bounded pseudo \(p\)-adically closed fields.
\end{abstract}

\maketitle

\input{Introduction}
\input{ImagAndCodes}
\input{PPC}

\sloppy
\bibliographystyle{alpha}
\bibliography{short,biblio}

\end{document}

%% file: Introduction.tex
\section*{Introduction}

Elimination of imaginaries is a positive answer to the question of finding definable moduli spaces for every definable family of definable sets; that is, given definable sets \(X\subseteq Y\times Z\), finding a definable map \(f : Y \to Z\) such that for all \(y_1,y_2\in Y\), \(f(y_1) = f(y_2)\) if and only if the fiber of \(X\) above \(y_1\) is equal to the one above \(y_2\). This is equivalent to the existence, for every set defined with parameters, of a smallest set of definition. In recent work of Hrushovski \cite{Hru-EIACVF} on the elimination of imaginaries in algebraically closed valued fields, the focus is shifted to the local question of finding smallest sets of definition for definable types --- and proving that there are enough definable types to deduce elimination of imaginaries.

But there are many structures of interest, among which the field of \(p\)-adic numbers, where there are too few definable types for this local approach to work. It is therefore tempting to work with invariant types instead. This presents a number of issues, the principal one being that, generally, invariant types cannot have smallest sets of definition due to their fundamentally infinitary nature. Some of them, for example, encode the cofinality of an infinite ordered set. In this paper, we choose to focus on a class of tractable invariant types: those that are arbitrarily close to definable types; the set \(\cDef(M/A)\) of \ref{approx inv}. The first part of this paper consists in providing the tools for an approach to elimination of imaginaries based on these definably approximable types.

The most important technical issue is the understanding of germs of definable functions at definably approximable types. At general invariant types, these germs are complicated hyperimaginaries. However, we show that, at a definably approximable type, germs of definable functions are encoded by the cofinality of a filtered ordered set of imaginary points. This can then used to encode germs of functions in a theory using elimination of imaginaries in a reduct, cf. \ref{inv acl}. We also give necessary and sufficient condition for every invariant type to be definably approximable --- equivalently, in an \(\NIP\) theory, for every non-forking formula to contain a definable type. In \(\dpr\)-minimal theories, Simon and Starchenko \cite{SimSta-DensDef} give a sufficient conditions for this density result to hold over models. We show that for the density result to hold over arbitrary algebraically closed basis, it suffices for the result to hold over models and for every definable set to contain a type definable over the algebraic closure of its code. In particular, the density result holds in algebraically closed valued fields (including imaginaries) --- cf. \ref{dens def Cmin}.

The other abstract contribution of this paper is to provide a very general criterion for weak elimination of imaginaries. Many proofs of elimination of imaginaries in tame unstable contexts follow the outline of Hrushovski's approach in bounded pseudo algebraically closed fields \cite[Proposition\,3.1]{Hru-PAC}. Thus, it seems interesting to isolate an abstract criterion pinpointing the exact ingredients of this proof. We show that if one can find an independence relation satisfying a strong form of extension and a certain amalgamation result, then weak elimination of imaginaries follows. For certain choices of independence relation (for example definable independence or invariant independence), this criterion reduces to previously known criteria of Hrushovski \cite[Lemma\,1.17]{Hru-EIACVF} or the second author \cite[Proposition\,9.1]{Rid-VDF}. For other choices of independence relation, one recovers a large number of previously known elimination of imaginaries proofs, for examples the one in algebraically closed fields with a generic automorphism \cite[\S 1.10]{ChaHru-ACFA}.\medskip

The second part of this paper is devoted to applying these general methods to the class of bounded pseudo \(p\)-adically closed fields. Note that the approach presented in this paper does not rely on the elimination of imaginaries in \(p\)-adic fields but rather reproves it, providing us with a new, slightly different, proof of \cite[Theorem\,2.6]{HruMarRid}. The new proof focuses, from the start, on constructing invariant types rather than reducing the problem to the elimination of unary imaginaries.

Before defining pseudo \(p\)-adically closed fields, let us recall that a field is said to be pseudo algebraically closed if every absolutely irreducible variety over this field has a rational point. This class of fields first appeared in Ax's work on pseudo finite fields \cite{Ax-Psf}. Their model theory has since been extensively studied, in parallel with the development of simplicity. Indeed bounded pseudo algebraically closed fields --- those that have only finitely many extensions of any given degree --- provide the main example of simple unstable fields. As interest in less restrictive tameness notions grew, for example notions like \(\NTP_{2}\) that do not preclude the existence of any definable order, it also became important to find algebraic examples, in particular enriched fields, that would provide us with study cases.

In \cite{Mon-EIPRC} and \cite{Mon-NTP2}, the first author thus started a neostability flavored study of two classes of large fields --- see \cite{Pop-Large} --- extending the class of pseudo algebraically closed fields: pseudo $p$-adically closed fields and pseudo real closed fields. Those two classes consist of the fields over which any absolutely irreducible variety with a simple point over every $p$-adically closed (respectively real closed) extension has a rational point. These classes were defined by van den Dries, Basarab, Prestel, Grob, Jarden and Haran, Ershov over 30 years ago in \cite{vdD-PhD, Bas-PRC, Pre, Gro, HarJar-PpC, Ers-RegCl}. A number of their model theoretic properties, pertaining to model completeness and the description of types, had been worked out. In \cite{Mon-NTP2}, the first author provided new tools to study these fields, mostly in the bounded case. She gives a description of definable sets --- the so-called density theorems \cite[Theorem\,3.17 and 6.11]{Mon-NTP2}  --- which generalizes cellular decomposition to the multi-valued (respectively multi-ordered) setting. She also proves an amalgamation theorem \cite[Theorem\,3.21 and 6.13]{Mon-NTP2} which is a weaker version of the independence theorem of simple theories which holds in this non-simple setting. She then uses those new tools to prove new classification results, the main one \cite[Theorem\,4.24 and 8.5]{Mon-NTP2} being that bounded pseudo $p$-adically closed and pseudo real closed fields are \(\NTP_2\) of finite burden.

One particularity that stands out in the first author's work is that, although most results are proved for both classes, elimination of imaginaries was only proved for pseudo real closed fields \cite{Mon-EIPRC}. The initial motivation for this paper was to repair that asymmetry. As it turns out, we were also able to repair a small gap in the proof of the pseudo real closed case, cf. \S\ref{S:PRC}.

Since bounded pseudo $p$-adically closed fields that are not pseudo algebraically closed fields come with finitely many definable valuations, one cannot expect elimination of imaginaries in a language with just one sort for the field, contrary to what happens with pseudo real closed fields. But since the work of \cite{HasHruMac-ACVF}, we know how to circumvent that particular issue: we have to add, for each valuation, codes for certain definable modules over the valuation ring; that is, work in the so-called geometric language. The main question regarding the imaginaries in bounded pseudo \(p\)-adically closed fields then becomes to prove that there are no imaginaries arising from the interaction between the various valuations and, therefore, that it suffices to add the geometric sorts for each of the valuations. The main result of this paper, \ref{EIPpC}, is a positive answer to this question. We deduce it from our abstract criterion, \ref{crit wEI}, applied to quantifier free invariant independence, cf \ref{qf ind}. The results of Section\,\ref{s:approx} play a fundamental role by allowing us to deduce \(n\)-ary extension from unary extension for that particular independence relation.

Our first step towards this elimination result, and a core ingredient of the rest of the paper, is \ref{orth geom} which states that not only are the geometric sorts for each valuation orthogonal but also that the structure of any given geometric sort is the one induced by the relevant \(p\)-adic closure.

We then proceed to deduce, from this strong statement on the independence of the valuations, a result on the structure of definable subsets of the valued field, where, at first sight, the valuations do interact. The key ingredient of the first author's proof that pseudo real closed fields eliminate imaginaries is \cite[Theorem 6.11]{Mon-NTP2} which states that, in dimension one, the decomposition given by the density theorem can be chosen in a canonical way. In pseudo \(p\)-adically closed fields, it is unclear if such a canonical decomposition exists. We choose instead to work at the level of types by proving, in \ref{loc dens higher}, that every type over algebraically closed sets of imaginary parameters is consistent with a global quantifier free type which is invariant over the geometric part of the parameters.

The last ingredient needed to prove weak elimination of imaginaries is to extend the first author's amalgamation result \cite[Theorem 6.13]{Mon-NTP2} to allow geometric parameters, cf. \ref{indep}. In particular, this requires us proving \ref{PpCExisClosed} which states an existential closedness property of pseudo \(p\)-adically closed fields in a stronger language than the field language, containing Macintyre's language for each \(p\)-adic valuation.

As is often the case, coding finite sets is mostly an independent issue. Here, we deduce it from the coding of finite sets in algebraically closed fields equipped with finitely many independent valuations. Our approach, inspired by Johnson's account \cite[\S 6.2]{Joh-EIACVF} of the coding of finite sets in algebraically closed valued fields, consists in first lifting tuples of geometric points by generically stable types of the valued field and then using the code of these types to encode finite sets.\medskip

The organization of the paper is as follows. Section\,\ref{S:im germs} contains the more abstract results and Section\,\ref{S:PpC} is devoted to pseudo \(p\)-adically closed fields. We start, in Section\,\ref{s:im prelim}, by recalling definitions related to imaginaries and their elimination. In Section\,\ref{s:approx}, we consider invariant types that can be approximated by definable types and the germs of definable functions over such types. We conclude the abstract part of the paper, in Section\,\ref{s:crit}, by proving a criterion for weak elimination of imaginaries and discussing various specific cases of that criterion.

Preliminaries on the model theory of valued fields and pseudo \(p\)-adically closed fields can be found in Section\,\ref{s:PpC prelim}. Section\,\ref{s:not} contains  various details regarding the language that we will be using. The orthogonality and purity of the geometric sorts is proved in Section\,\ref{s:orth geom}. Section\,\ref{s:acl} is devoted to the description of the algebraic closure in bounded pseudo \(p\)-adically closed fields, including geometric imaginaries, and in Section\,\ref{s:dens}, we prove the existence of quantifier free invariant extensions of types over algebraically closed bases. Amalgamation over geometric points is proved in Section\,\ref{s:amalg}. Finally, in Section\,\ref{s:finite}, we code finite sets and we deduce elimination of imaginaries in Section\,\ref{s:main}.\medskip

The authors would like to thank Gabriel Conant for giving them the idea of working with an abstract independence relation in the criterion for weak elimination of imaginaries.

%% file: ImagAndCodes.tex
\section{Imaginaries and germs of functions}\label{S:im germs} 

\subsection{Preliminaries}\label{s:im prelim}
 
We will assume knowledge of standard model theoretic knowledge and notation. We refer the reader to \cite{TenZie} for an introduction to model theory.
 
Let us start by recalling the basic definitions regarding elimination of
imaginaries. The reader can look at \cite{Poi-ImGal} for more details. Let $X$ be a definable set in some $\cL$-structure
$M$. The tuple $b\in M$ is a canonical parameter for $X$ (via the
$\cL$-formula $\phi(x, y)$) if for all tuples $b'\in M$, $X = \phi(M,b')$ if
and only if $b' = b$. An $\cL$-theory $T$ is said to eliminate
imaginaries if every definable set in any model of $T$ has a canonical
parameter (via some $\cL$-formula). Equivalently, if $T$ has
sufficiently many constants, $T$ eliminates imaginaries if and only
if, for every $\cL$-formula $\phi(x, y)$, there exists an
$\emptyset$-definable map $f$ such that
\[T\vdash \forall y_1\forall y_2\,(f(y_1) = f(y_2)\iffform (\forall
  x\,(\phi(x, y_1)\iffform\phi(x, y_2)))).\] Note that if \(\phi\) defines an equivalence relation $E$ in \(T\), then we have
that $T\vdash \forall x\forall y\,(x E y\iffform f(x) = f(y))$.

Given any language \(\cL\), we define the language \(\eq{\cL}\) that contains, for each $\cL$-formula $\phi(x, y)$, a new sort $E_\phi$ and a new function
symbol $f_\phi$ from the product of sorts of \(y\) to \(E_{\phi}\). Given an \(\cL\)-theory \(T\), we define the \(\eq{\cL}\)-theory \[\eq{T} := T\cup \bigcup_\phi\{f_\phi\text{ is onto and }\forall y_1\forall y_2\,(f(y_1) = f(y_2)\iffform (\forall
x\,(\phi(x, y_1)\iffform\phi(x, y_2))))\}.\]
The theory $\eq{T}$
eliminates imaginaries and to any $M\models T$, we can
associate a unique $\eq{M}\models\eq{T}$ whose reduct to $\cL$ is $M$. We denote by \(\acleq\) (respectively \(\dcleq\)) the algebraic (respectively definable) closure in \(\eq{M}\). If \(X\) is an \(\cL(M)\)-definable set, we define \(\code{X} = \dcleq(b)\subseteq \eq{M}\) for any choice of canonical parameter \(b\) of \(X\).

\subsection{Invariant types approximated by definable types}\label{s:approx}

In this section, we want to show how being able to approximate invariant types by definable
types can be helpful to compute "canonical bases" of invariant types. This will allow us to study the invariance of germs of functions over invariant types. In \ref{crit dens def}, we also give necessary and sufficient conditions, that hold in algebraically closed valued fields, for invariant types to be approximated by definable types.

We denote by $\TP(M)$ the type space over $M$ and by $\TP_x(M)$ the type space over $M$ in variable $x$.  

\begin{definition}[approx inv]
  Let \(T\) be some \(\LL\)-theory and \(A \subseteq M
  \models T\) and $p(x) \in \TP_x(M)$.
  \begin{enumerate}
  \item We say that $p$ is \(\aut(M/A)\)-invariant if for all $\sigma
    \in \aut(M/A)$, any $\cL$-formula $\phi(x, y)$ and tuple $a\in A$,
    $\phi(x, a)\in p$ if and only if $\phi(x,\sigma(a))\in p$.
  \item The type $p$ is $\cL(A)$-definable if for every
    $\cL$-formula $\phi(x, y)$, there is some $\cL(A)$-formula
    $\defsc{p}{x}\,\phi(x, y) = \theta(y)$ such that for all $a \in A$,
    $p \vdash \phi(x, a)$ if and only if $M \models
    \defsc{p}{x}\,\phi(x,a)$.
  \end{enumerate} 
\end{definition}

\begin{definition}
  Let \(T\) be some \(\LL\)-theory and let \(A \subseteq M \models
  T\). We define:
\begin{enumerate}
\item \(\Inv_{x}(M/A)\subseteq \TP_x(M)\) the set of
  \(\aut(M/A)\)-invariant types over $M$;
\item \(\Def_x(M/A)\subseteq \TP_x(M)\) the set of
  \(\LL(A)\)-definable types over $M$;
\item \(\cDef_x(M/A)\) the set of types \(p\in\TP_x(M)\) such
  that for all \(\LL(M)\)-formula \(\phi(x)\in p\), there exists
  \(q\in\Def_x(M/A)\) such that \(\phi\in q\).
\end{enumerate}
\end{definition}

\begin{remark}
  We have
  \(\Def_x(M/A)\subseteq\cDef_x(M/A)\subseteq\Inv_x(M/A)\)
  and the last two sets are closed in \(\TP_x(M)\). Moreover,
  as the notation indicates, \(\cDef_x(M/A)\) is the closure of
  \(\Def_x(M/A)\).
\end{remark}

\begin{fact}[approx type]
  Let $T$ be some $\cL$-theory, let \(A \subseteq M \models
  T\) and let \(p\in\cDef_x(M/A)\). Then there exists a sequence
  \((q_i)_{i\in I} \in\Def_x(M/A)\), indexed by a directed set of size
  at most \(2^{|\LL(A)|}\), such that for all \(\LL(M)\)-formula
  \(\phi(x)\), \(\phi\in p\) if and only if \(\phi\in q_i\) for almost
  all \(i\).
\end{fact}

By almost all \(i\), we mean that it holds for all \(i\) greater that
some \(i_0\in I\). This is just the characterization of closure by
nets.

\begin{definition}
  Let \(M\) be an \(\LL\)-structure \(p\in\TP_x(M)\) and \(f\), \(g\) be
  \(\LL(M)\)-definable functions, defined at \(p\). We say that \(f\)
  and \(g\) have the same \(p\)-germ if \(p(x) \vdash f(x) =
  g(x)\). We denote by \(\germ{p}{f}\) the class of \(f\) for this
  equivalence relation.
\end{definition}

A priori, \(\germ{p}{f}\) is an hyperimaginary: an equivalence class for an invariant equivalence relation. 
If \(p\) is definable, it can be identified with an imaginary point. The core of the
following proposition is that, if \(p\in\cDef(M/A)\), then the
hyperimaginary \(\germ{p}{f}\) is coded by the "limit" of a sequence
of imaginaries.\medskip

In what follows, let \(T_0\) be some \(\LL_0\)-theory eliminating
quantifiers and imaginaries whose sorts we denote \(\Real_0\). Let $T$
be a complete \(\cL\)-theory containing the universal part of $T_0$. When applying these results to pseudo $p$-adically closed fields, we take \(T\) to be the complete theory of some bounded pseudo $p$-adically closed field and \(T_0\) to be the theory of its algebraic closure with (the extension of) one of its valuations.

Let \(M\models T\) be sufficiently saturated and homogeneous and
\(M_0\models T_0\) containing \(\Real_0(M)\) be such that any
automorphism of \(M\) extends to \(M_0\). When we say that something is \(\cL_0\)-definable, it will mean that it is definable in \(M_0\). Assume that:
\begin{itemize}
\item[($\dagger$)] for all \(\epsilon\in\dcl^{M_0}_{\LL_0}(\Real_0(M))\),
  there exists a tuple \(\eta\in\Real_0(M)\) such that \(\epsilon\)
  and \(\eta\) are interdefinable in the pair \((M_0,M)\).
\end{itemize}

\begin{proposition}[code germ cdef]
 Let \(A\subseteq \eq{M}\) be such that \(\Real_0(\acl[\LL](A))\subseteq A\), \(p\in\cDef_x^{\LL_0}(M_0/\Real_0(A))\) and \(f\) be an
  \(\LL_0(\Real_0(M))\)-definable function defined at \(p\). If
  \(\germ{p}{f}\) has a finite orbit under \(\aut[\LL](M/ A)\), then
  it is \(\aut[\LL_0](M_0/\Real_0(A))\)-invariant.
\end{proposition}

\begin{proof}
  Let \((q_i)_{i\in I}\in\Def_x^{\cL_0}(M_0/\Real_0(A))\) be as in
  \ref{approx type} with respect to \(p\) and let \(\epsilon_i
  :=\germ{q_i}{f} \in M_0\) and \(\eta_i\) be as in Hypothesis\,\((\dagger)\)
  with respect to \(\epsilon_i\). Let \(F = (F_{m})_{m}\) be an
  \(\LL_0\)-definable family of functions such that \(f = F_{m_0}\)
  for some \(m_0\). Note that \(p(x)\vdash F_{m_1}(x) = F_{m_2}(x)\)
  if and only if \(q_i\vdash F_{m_1}(x) = F_{m_2}(x)\) for almost all
  \(i\). In particular, for any
  \(\sigma\in\aut[\cL_0](M_0/\Real_0(A))\),
  \(\sigma(\germ{p}{f}) = \germ{p}{f}\) if and only if
  \(\sigma(\epsilon_i) = \epsilon_i\) for almost all \(i\).

\begin{claim}
\(\eta_i\in \Real_0(A)\), for almost all \(i\in I\).
\end{claim}

\begin{proof}
  As \(\eta_i\in\Real_0(M)\), it suffices to prove that \(\eta_i\in
  A\) for almost all \(i\). If not, there is an unbounded subset
  $J\subseteq I$ such that for all \(j\in J\), \(\eta_{j}\) has an
  infinite $\aut[\LL](M/A)$-orbit and hence
  \(\aut[\LL](M/A\eta_j)\) has infinite index. By Neumann's
  lemma, for all choice of \(j_1,\ldots, j_n \in J\), there exists
  \(\tau_1,\ldots\tau_n\in\aut[\LL](M/A)\) such that, for all \(k\),
  the \(\tau_l(\eta_{j_k})\) are all distinct. By saturation and
  homogeneity, there exists \((\tau_l)_{l\in\omega}\in\aut[\LL](M/A)\)
  such that, for all \(j\in J\), the \(\tau_l(\eta_j)\) are all
  distinct and hence, extending \(\tau_l\) to \(M_0\), so are all the $\tau_l(\germ{q_j}{f})$. Since the
  set $J$ is unbounded, it follows that the \(\tau_l(\germ{p}{f})\) are
  all distinct, a contradiction.
\end{proof}

We have proved that \(\aut[\LL_0](M_0/\Real_0(A))\) fixes almost
all \(\epsilon_i\) and hence it fixes \(\germ{p}{f}\).
\end{proof}

Let \(\cH\) be a subset of \(\Real_0\). We now also assume that:
\begin{itemize}
\item[($\star$)] For all tuple \(a\in\cH(N)\), where \(N\supsel M\),
  \(\Real_0(\acl_{\cL}(Ma))\subseteq \acl_{\cL_0}(\Real_0(Ma))\).
\end{itemize}

In our application to pseudo \(p\)-adically closed fields, the sort \(\cH\) will be the field sort.

\begin{proposition}[inv acl]
  Let \(a \in \cH(N)\), where \(N\supsel M\), \(A\subseteq \eq{M}\) be such that \(\Real_0(\acl[\LL](A))\subseteq A\), and \(c\in\Real_0(\eq{\acl_{\cL}}(Aa))\) be tuples. 
  Assume that \(p:= \tp_{\LL_0}(\Real_0(a)/M_0)\in\cDef(M_0/\Real_0(A))\)  Then, there
  exists an \(\cL_0(M)\)-definable map \(F\) such that \(\germ{p}{F}\)
  is \(\aut[\LL_0](M_0/\Real_0(A))\)-invariant and
  \(c\in\acl_{\cL_0}(F(a))\).
\end{proposition}

\begin{proof}
Since \(c\in\eq{\acl_{\cL}}(Aa)\), we can find an
\(\eq{\cL}(A)\)-definable function \(f\) such that \(f(a)\) encodes a
finite set containing \(c\). By (\(\star\)), we have that
\(f(a)\subseteq \Real_0(\acl_{\cL}(Ma)) \subseteq
\acl_{\LL_0}(Ma)\). It follows that we can find an
\(\cL_0(M)\)-definable function \(F\) such that \(F(a)\) encodes (in
\(M_0\)) a finite set containing \(f(a)\). Note that we have \(c \in
f(a) \subseteq \acl_{\cL_0}(F(a))\).

By compactness (and replacing \(F\) with a finite union), we may
assume that \(f(x)\subseteq F(x)\) always holds. The cardinality of
\(F\) is constant on realizations of \(p\). We may assume it is
minimal among all possible \(F\). Let \(\sigma\in\aut[\cL](M/A)\),
then \(f(a) \subseteq F(a) \cap F^{\sigma}(a)\). By minimality, we
must have \(|F(a)\cap F^{\sigma}(a)| = |F(a)| = |F^{\sigma}(a)|\) and
hence \(F(a) = F^{\sigma}(a)\). We have just proved that
\(\germ{p}{F}\) is \(\aut[\cL](M/A)\)-invariant. By \ref{code germ
  cdef}, it is, in fact, \(\aut[\cL_0](M_0/\Real_0(A))\)-invariant.
\end{proof}

The conclusion of \ref{inv acl} is weaker than one might expect. It does not imply, in general, that \(c\in\acl_{\cL_0}(\Real_0(A)a)\). However, if germs of functions over \(p\) are particularly well-behaved, this can be the case. If \(p\in \TP_{\cL_0}(M_0)\) is \(\cL_0(A)\)-definable and \(f\) is \(\cL_0(M_0)\)-definable and defined at \(p\), we say that \(\germ{p}{f}\) is strong if there exists an \(\cL_0(A\germ{p}{f})\)-definable map \(g \in \germ{p}{f}\). This happens whenever \(p\) is generically stable, see  \cite{AdlCasPil}; in particular, if \(T_0\) is stable. In that case, it becomes quite clear that what we have been doing so far is indeed encode (germs of) functions:

\begin{corollary}[gen stab acl]
In the setting of \ref{inv acl}, assume \(p\) is definable and germs of functions over \(p\) are strong. Then \[\Real_0(\eq{\acl_{\cL}}(Aa)) \subseteq \acl_{\cL_0}(\Real_0(A)a).\]
\end{corollary}

\begin{proof}
Pick any \(c\in \Real_0(\eq{\acl_{\cL}}(Aa))\) and let \(F\) be as in \ref{inv acl}. Since \(p\) is definable, we have that \(\germ{p}{F}\in\dcl_{\cL_0}(\Real_0(A)) \subseteq \Real_0(A)\). Also, since \(\germ{p}{F}\) is strong, we may assume that \(F\) is \(\cL_0(\Real_0(A))\)-definable. Thus, \(c\in\acl_{\cL_0}(F(a)) = \acl_{\cL_0}(G(a)) \subseteq \acl_{\cL_0}(\Real_0(A)a)\).
\end{proof}

For technical reasons, we will later need to involve a third intermediary language.
Let \(\cL_0\subseteq \cL_1\subseteq\cL\), \(T_{0,\forall}\subseteq
T_{1,\forall}\subseteq T\) and \(M \subseteq M_1\subseteq M_0\) be such
that \(T_1\) eliminates quantifiers, \(M_1\models T_1\) and every automorphism of \(M_1\) extends to
an automorphism of \(M_0\). In our application to pseudo $p$-adically closed fields, \(M_1\) will be the \(p\)-adic closure of our pseudo -$p$-adically closed field \(M\) with respect to one of its valuations, in Macintyre's language. Recall that \(M_0\) will be the algebraic closure of \(M\) with an extension of that same valuation. The reason we need to consider those three structures is that \(M_1\) is the one with the right quantifier free structure induced on \(M\), but \(M_0\) is the one with sufficiently many definable types.\medskip

However, the following corollary retains most of its content when \(T_0 = T_1\) and the reader should feel free to first consider that case. Not that even when \(T_0 = T_1 = T\), the statement is non-trivial and not known to be true in general. It is known, however, when \(T_0 = T_1 = T\) is \(\NIP\) by work of Hrushovski and Pillay \cite[Lemma\,2.12]{HruPil-NIP}.

\begin{corollary}[inv acl NIP]
 Assume \(T_1\) is \(\NIP\). Let \(N\supsel M\) and \(a \in \cH(N)\) be a tuple. If \(p:= \tp_{\LL_0}(\Real_0(a)/M_0)\in\cDef(M_0/\Real_0(A))\) and \(q:= \tp_{\LL_1}(\Real_0(a)/M_1)\in\Inv(M_1/\Real_0(A))\), then \(\tp_{\LL_1}(\Real_0(\eq{\acl_{\cL}}(Aa))/M_1)\in\Inv(M_1/\Real_0(A))\).
\end{corollary}

\begin{proof}
By \ref{inv acl}, we can find \(F\) such that \(\germ{p}{F}\) is
\(\aut[\LL_0](M_0/\Real_0(A))\)-invariant and
\(\Real_0(\eq{\acl_{\cL}}(Aa))\subseteq\acl_{\cL_0}(F(a))\subseteq\acl_{\cL_1}(F(a))\). Since
automorphisms of \(M_1\) extend to automorphisms of \(M_0\),
\(\germ{q}{F}\) is \(\aut[\LL_1](M_1/\Real_0(A))\)-invariant and hence \(\tp_{\LL_1}(F(a)/M_1)\in\Inv(M_1/\Real_0(A))\). By \cite[Lemma\,2.12]{HruPil-NIP}, we now have that \(\tp_{\cL_1}(\acl_{\cL_1}(F(a))/M_1)\in\Inv(M_1/\Real_0(A))\). In particular, \(\tp_{\LL_1}(\Real_0(\eq{\acl_{\cL}}(Aa))/M_1)\in\Inv(M_1/\Real_0(A))\).
\end{proof}

Note that, in the above corollary, we only obtain an invariant type and not, \emph{a priori}, a definably approximable one. Let us conclude with examples of theories in which \(\Inv(M/A) = \cDef(M/A)\), in which case we will be able to continue with the induction and build invariant types in all arities, cf. \ref{loc dens higher}.

\begin{proposition}[crit dens def]
  Let \(T\) be an \(\LL\)-theory that weakly eliminates imaginaries
  and \(M\) be sufficiently saturated. Assume:
  \begin{enumerate}[label=(\roman*),ref=(\roman*)]
  \item\label{strong dens def} for all \(N\subsel M\models T\),
    \(\Inv(M/N) = \cDef(M/N)\);
  \item\label{weak dens def} for all \(A=\acl(A)\subseteq M\models T\),
    any unary $\LL(A)$-definable set is consistent with a
    type $p\in\Def(M/A)$.
  \end{enumerate}
  Then, for all \(A=\acl(A)\subseteq M\models T\), \(\Inv(M/A) =
  \cDef(M/A)\).
\end{proposition}

Note that the converse is obvious.

\begin{proof}
Let us start with some classic results.

\begin{claim}[def typ acl]
Let \(A=\acl(A)\subseteq M\). Assume \(\tp(a/M)\in\Def(M/A)\) and \(c
\subseteq \acl(Aa)\), then \(\tp(c/M)\in\Def(M/A)\).
\end{claim}

\begin{proof}
We may assume that \(c\) is a finite tuple. Let \(\phi(y, x)\) be an
\(\cL(A)\)-formula such that \(\phi(y, a)\) algebrizes \(c\) over
\(Aa\). Let \(p = \tp(a/M)\) and pick \(\psi(y, z)\) any
\(\cL\)-formula. We define \(z_1 E z_2\) to hold if
\(M\models\defsc{p}{x}\,(\forall y\,\phi(y, x) \impform
(\psi(y, z_1)\iffform \psi(y, z_2)))\). This is an \(\cL(A)\)-definable
finite equivalence relation and its classes are \(\cL(A)\)-definable by weak elimination of imaginaries. Moreover, \(\phi(c,M)\) is a finite union
of \(E\)-classes. It is therefore \(\cL(A)\)-definable.
\end{proof}

\begin{claim}[trans def]
Pick \(A\subseteq N\subsel M\), \(p\in\Def(N/A)\), \(a\models p\) in \(M\),
\(q\in\Def(M/Aa)\) and \(c\models q\). Then \(\tp(ac/N)\in\Def(N/A)\).
\end{claim}

\begin{proof}
Pick \(\psi(x, y, z)\) any \(\cL\)-formula. Let \(\theta(y, z, a)\) be an
\(\cL(A)\)-formula such that, for all \(b\), \(d\in M\),
\(\psi(b, y, d)\in q\) if and only if \(M\models\theta(b, d, a)\). Then,
for all \(d\in N\), \(M\models\psi(a, c, d)\), if and only if \(M\models\theta(a, d, a)\), if and only if
\(M\models\defsc{p}{x}\,\theta(x, d, x)\).
\end{proof}

\begin{claim}[def model]
For all \(N\subsel M\) containing \(A\), there exists \(N^\star
\subsel M\) containing \(A\) such that \(\tp(N^\star/N)\in\Def(N/A)\).
\end{claim}

\begin{proof}
  Let \(\{\phi_i(x_i)\mid i\in\kappa\}\) be an enumeration of all
  consistent \(\LL(A)\)-formulas in one variable $x_i$. By induction,
  for all \(i\), we build \(a_i\in M\) such that
  \(M\models\phi_i(a_i)\), \(c_i = \acl(Ac_{<i}a_i)\) and
  \(\tp(c_i/N)\in\Def(N/A)\). The element \(a_i\) is found by
  Hypothesis\,\ref{weak dens def} applied over \(Ac_{<i} =
  \acl(Ac_{<i})\). Then by \ref{trans def},
  \(\tp(c_{<i}a_i/N)\in\Def(N/A)\). Let \(c_i = \acl(Ac_{<i}a_i)\). By
  \ref{def typ acl}, \(\tp(c_i/N)\in\Def(N/A)\).

  Let \(d_0 = A\) and \(d_1 = c_{<\kappa} = \acl(A d_1)\). By
  repeating the above construction, we obtain \((d_i)_{i\in\omega}\)
  such that, for all \(i\in\omega\),
  \(\tp(d_i/N)\in\Def(N/A)\) and any consistent
  \(\LL(d_{<i})\)-formula in one variable is realized in \(d_i\). Let \(N^\star\) be the set enumerated by \(d_{<\omega}\). By
  Tarski-Vaught, \(N^\star\subsel M\). Moreover, \(A\subseteq
  N^\star\) and \(\tp(N^\star/N)\in\Def(N/A)\).
\end{proof}

Pick \(p\in\Inv(M/A)\). We have to show that for all
\(\LL(M)\)-formula \(\phi(x)\), if \(p\vdash \phi\), then there exists
\(q\in\Def(M/A)\) such that \(q\vdash \phi\). Actually, it suffices to
find such a \(q\in\Def(N/A)\) for some \(N\subsel M\) containing \(A\)
and such that \(\phi\in\LL(N)\). By \ref{def model}, we find
\(A\subseteq N^\star\subseteq M\) such that
$\tp(N^\star/N)\in\Def(N/A)$. By Hypothesis\,\ref{strong dens def}, we
find \(q\in\Def(M/N^\star)\) consistent with \(\phi\). Let \(c\models
q\), by \ref{trans def}, \(\tp(N^\star c/N)\in\Def(N/A)\). In
particular, \(\phi\in \restr{q}{N} = \tp(c/N)\in\Def(N/A)\).
\end{proof}

\begin{remark}[unary]
  Let $\cH$ be a set of dominant sorts in $T$, i.e. any other sort is
  the image of a $\emptyset$-definable function whose domain is a
  product of sorts in $\cH$. Then, in \ref{crit dens def}, it suffices
  to assume Hypothesis\,\ref{weak dens def} for every definable subsets of every sort in \(\cH\).
\end{remark}

\begin{corollary}[dens def Cmin]
  Let \(T\) be any \(C\)-minimal theory which
  weakly eliminates imaginaries. Then, for all \(A = \acl(A)\subseteq
  M\models T\) sufficiently saturated, \(\Inv(M/A) = \cDef(M/A)\).
\end{corollary}

\begin{proof}
Hypothesis\,\ref{weak dens def} of \ref{crit dens def} holds for definable subsets of \(\K\)
in \(C\)-minimal theories by taking the generic type of any outer ball
of the Swiss cheese decomposition. In \cite[Theorem\,2.3]{SimSta-DensDef}, Simon and Starchenko show that for \(\dpr\)-minimal theories (in particular, \(C\)-minimal theories), Hypothesis\,\ref{strong dens def} of \ref{crit dens def} follows from Hypothesis\,\ref{weak dens def}. Note that, by \cite[Lemma\,2.4]{SimSta-DensDef}, Hypothesis\,\ref{weak dens def} does imply their property (D).
\end{proof}

\subsection{A criterion using amalgamation}\label{s:crit}

The following criterion is an attempt at an abstract account of
the proof given by Hrushovski in \cite[Proposition\,3.1]{Hru-PAC} and adapted in
many various settings since then. Let $T$ be an $\cL$-theory with
sorts $\Real$ and $M\models T$ be sufficiently saturated and
homogeneous.

\begin{proposition}[crit wEI]
  Let $\indep$ be a ternary relation on small subsets of $M$. Assume:
  \begin{enumerate}[label=(\roman*),ref=(\roman*)]
  \item for all $E = \acleq(E)\subseteq \eq{M}$, tuple
    $a\in M$ and $C\subseteq M$, there exists $a^\star$ such that
    $a^\star \equiv_{\cL(E)} a$ and $a^\star\indep_{\Real(E)} C$;
  \item\label{H:amalg} for all $E = \acl(E)\subseteq M$ and tuples
    $a, b, c\in M$, if $b\indep_E a$, $c\indep_E ab$ and
    $a\equiv_{\LL(E)} b$, then there exists $c^\star$ such that $ac
    \equiv_{\LL(E)} ac^\star \equiv_{\LL(E)} bc^\star$.
  \end{enumerate}
  Then $T$ weakly eliminates imaginaries.
\end{proposition}

\begin{proof}
  Pick any $e\in\eq{M}$. There exists an $\emptyset$-definable map $f$
  and a tuple $a\in M$ such that $e = f(a)$. Let $\overline{E} =
  \acleq(e)$ and $E = \Real(\overline{E})$. We want
  $e\in\dcleq(E)$. It suffices to prove that $\tp(a/E)\vdash f(a) =
  e$.
  
  Pick any $b \equiv_{\cL(E)} a$. By Hypothesis\,(i), there exists
  $b^\star\equiv_{\cL(\overline{E})} b $ such that $b^\star \indep_E
  a$. Applying Hypothesis\,(i) again, we find
  $c\equiv_{\cL(\overline{E})} a$ such that $c\indep_{E} ab$. Note
  that we have $f(c) = e = f(a)$. Now, applying Hypothesis\,(ii), we
  find $c^\star$ such that
  $ac\equiv_{\cL(E)}ac^\star\equiv_{\cL(E)}b^\star c^\star$. It
  follows that $e = f(a) = f(c) = f(c^\star) = f(b^\star)$. Since
  $b^\star\equiv_{\cL(\overline{E})} b $ and $e\in\overline{E}$, we
  have that $f(b) = e$.
\end{proof}

The notion of independence that we will be using, in this paper, is quantifier free
invariant independence.

\begin{definition}[qf ind]
  \begin{enumerate}
  \item Let $A\subseteq M$. A quantifier free type $p$ over $M$ is said to
    be $\aut(M/A)$-invariant if for every quantifier free
    $\cL$-formula $\phi(x, y)$, $a\in M$ and $\sigma\in\aut(M/A)$,
    $\phi(x, a)\in p$ if and only if $\phi(x,\sigma(a))\in p$.
  \item Let $a \in M$ be a tuple and $C,B\subset M$. We define
    $a\qfindep_C B$ to hold if there exists an $\aut(M/A)$-invariant
    quantifier free type $p$ over $M$ such that $a\models
    \restr{p}{CB}$.
  \end{enumerate}
\end{definition}

We write \(a \qfequiv[\cL(E)] b\) to say \(a\) and \(b\) have the same quantifier free type over \(E\).

\begin{lemma}[indep suff]
  Hypothesis\,\ref{H:amalg} of \ref{crit wEI} for $\qfindep$ follows from:
  \begin{enumerate}
  \item[(ii')] for all $E = \acl(E)\subseteq M$ and
    tuples $a_1,a_2,c_1,c_2,c\in M$, if $a_1\qfindep_E a_2$,
    $c\qfindep_E a_1a_2$ and $c_1\equiv_{\LL(E)} c_2$ and, for all
    $i$, $a_ic_i \qfequiv[\cL(E)] a_i c$, then there exists $c^\star$
    such that $a_ic_i \equiv_{\LL(E)} a_ic^\star$, for all $i$.
  \end{enumerate}
\end{lemma}

\begin{proof}
  Let $E$, $a$, $b$ and $c$ be as in Hypothesis\,\ref{H:amalg}. Since
  $a\equiv_{\cL(A)} b$, we can find $d$ such that $ac\equiv_{\cL(E)}
  bd$. Then we do have $a \qfindep b$, $c\qfindep_E ab$ and
  $c\equiv_{\LL(E)} d$. Moreover, since $a$ and $b$ are
  $\aut(M/E)$-conjugated and there exists an $\aut(M/E)$-invariant
  quantifier free type $p$, it follows that $a\qfequiv[\cL(Ec)] b$ and
  hence $bd \qfequiv[\cL(E)] bc$. It follows that we can apply (ii')
  to find a $c^\star$ that $ac \equiv_{\LL(E)} ac^\star$ and $bd
  \equiv_{\LL(E)} bc^\star$. Since $bd \equiv_{\cL(E)} ac$, we have
  the required conclusion.
\end{proof}

\begin{remark}[QE crit]
Note that if $T$ eliminates quantifiers (or if we are working in the Morleyized language), (ii') and therefore (ii), holds trivially for $\qfindep$. In that case, \ref{crit wEI} reduces to the statement that if every set \(X\) definable in some model of \(T\) contains a type which is \(\Real(\acleq(\code{X}))\)-invariant, then \(T\) weakly eliminates imaginaries. A particular case of that statement is Hrushovski's criterion via definable types \cite[Lemma\,1.17]{Hru-EIACVF} where invariant is replaced by definable --- and the coding of definable types is made explicit.

One important difference of working with definable types, though, is that, as pointed out in \ref{unary}, building types on some \(X\) which are definable almost over \(\code{X}\), without controlling in which sorts the canonical basis lies, can be done just for unary dominant definable sets. The general case follows by an easy induction. Also, contrary to definable types, invariant types do not have an imaginary canonical basis and thus it is much harder to compute their canonical basis, unless it is obvious from the construction of the invariant type, as in \ref{loc dens}.
\end{remark}


%% file: PPC.tex
\section{\texorpdfstring{Pseudo $p$-adically closed fields}{Pseudo p-adically closed fields}}
\label{S:PpC}
\subsection{Preliminaries}
\label{s:PpC prelim}

In this section we will give all the preliminaries on valued fields,
$p$-adically closed fields and pseudo $p$-adically fields that
are required throughout the rest of the paper.

\begin{notation}
 Whenever $F$ is a field, we denote by $\alg{F}$ its algebraic
  closure.
\end{notation}


Let us start by fixing our notations regarding valued fields and introduce the geometric language.

\begin{definition}
A valuation on a field $F$ is a group morphism map $\val: \inv{F} \rightarrow \Gamma$, where $(\Gamma, +,0, <)$ is a totally ordered abelian group, such that for all $a, b \in F$, $\val(a+b)\geq \min\{\val(a), \val(b)\}$. The set $\Val:= \{x \in F: \val(x)\geq 0\}$ is a subring of $F$ that
is called the valuation ring.  The set $\Mid:= \{x \in F: \val(x)>
0\}$ is the unique maximal ideal of $\Val$. The field $\res := \Val/
\Mid$ is called the residue field of $F$.
\end{definition}

\begin{definition}(Geometric language)
Let $(F,\val)$ be a valued field. For all $m\in\Zz_{>0}$, We define \[\Latt{m}(F) :=
      \GL{m}(F)/\GL{m}(\Val)\text{ and }\Tor_m(F) :=
      \GL{m}(F)/\GL{m, m}(\Val),\] where $\GL{m, m}(\Val)$ is the group
      of matrices $M\in \GL{m}(\Val)$ whose last column reduces modulo
      $\Mid$ to a column of zero except for a $1$ on the diagonal.

The \emph{geometric language} $\LG$ consists of
      the sort $\K$ equipped with the ring language, sorts $\Latt{m}$
      and $\Tor_m$, for all $m \in \Zz_{>0}$, maps $s_m: \K^{m^2}
      \rightarrow \Latt{m}$ and $t_m: \K^{m^2}\rightarrow \Tor_m$
      interpreted as the canonical projections and the necessary predicates to have quantifier elimination in the \(\LG\)-theory \(\ACVFG\) of algebraically closed valued fields.
\end{definition}

\begin{remark}[rem G]
  \begin{enumerate}[label=(\arabic*),ref=(\arabic*)]
  \item The sort $\Latt{m}$ is the moduli space of all rank $m$ free $\Val$-submodules of $\K^m$, i.e. $\Val$-lattices. For all \(s\in\Latt{m}\), let \(\latt(s)\) to be the lattice whose set of bases is \(s\).
  \item We can identify $\Latt{1} = \K/\inv{\Val}$ with $\val(\inv{\K})$
    and $s_1$ with the valuation map.
  \item For all \(s\in \Latt{m}\), the fiber of the natural map $\tau_m : \Tor_m \to \Latt{m}$ above \(s\) can be identified (once we add a zero) with the dimension \(m\) $\res$-vector space \(\latt(s)/\Mid\latt(s)\).
  \item\label{Tn non nec} When the valuation $\val$ is discrete, i.e. there are elements $\theta$ with minimal positive valuation, there is an $\emptyset$-definable
    map from $\Tor_m$ to $\Latt{m+1}$. It follows that, in that case,
    the sorts $\Tor_m$ are not necessary to obtain elimination of
    imaginaries.
  \end{enumerate}
\end{remark}

The geometric sorts where introduced by Haskell, Hrushovski and Macpherson to prove:

\begin{theorem}({\cite{HasHruMac-ACVF}})
  The theory $\ACVFG$ of algebraically closed valued fields in the
  geometric language eliminates imaginaries.
\end{theorem}

Let us now discuss \(p\)-adically closed fields.

\begin{definition}
  Let $(F,\val)$ be a valued field.
  \begin{enumerate}
  \item The valuation $\val$ is called $p$-adic if the residue field
    is $\Ff_p$ and $\val(p)$ is the smallest positive element of the
    value group $\val(\inv{F})$.
  \item We say that $(F,\val)$ is $p$-adically closed if $(F,\val)$ is
    $p$-adically valued and it has no proper $p$-adically valued
    algebraic extension.
  \item A $p$-adic closure of $(F,\val)$ is an algebraic extension
    $(L,\val)$, which is $p$-adically closed. It
    always exists when $(F,\val)$ is $p$-adically valued.
  \end{enumerate}
\end{definition}

\begin{fact}(Properties of the theory of $p$-adically closed fields)
  \begin{enumerate}[label=(\roman*)]
  \item The class $\pCF$ of $p$-adically closed fields is elementary
    in the language $\Lrg := \{+,-, \cdot, 0,1\}$ of rings.
  \item Let $\LP$ be the language of rings to which we add a binary
    predicate $\Div$ and, for all $m\in\Zz_{>0}$, unary predicates
    $\{\Pow{m}: m > 1\}$. We interpret $x\Div y$ as $\val(x)\leq
    \val(y)$ and $P_m$ as set of non-zero $m$-th powers. By \cite[Theorem\,1]{Mac-Qp}, $\pCF$ eliminates quantifiers in $\LP$.
  \end{enumerate}
\end{fact}

\begin{remark}
Observe that \(x \Div y\) is quantifier-free definable in \(\LP\sminus\{\Div\}\):

\[M \models a \Div b \text{ if and only if } M \models P_\ell(a^\ell + p b^\ell),\] where $\ell\neq p$ is prime.
\end{remark}


We now define pseudo $p$-adically closed fields and recall known results. We end this section by proving a new result regarding existential closedness of $p$-adically closed fields with finitely many distinct \(p\)-adic valuations. We
refer the reader to \cite{Jar-PpC}, \cite{HarJar-PpC} and
\cite{Mon-NTP2} for more details.

\begin{definition}
Let $F\leq L$ be an extension of characteristic $0$ fields. 
\begin{enumerate}
\item $F\leq L$ is called totally $p$-adic if every $p$-adic valuation of $F$ can be extended to a $p$-adic valuation of $L$.
 \item $F\leq L$ is a regular extension if $L \cap \alg{F}=F$. 
\end{enumerate}
\end{definition}

\begin{definition}
A field $F$ of characteristic $0$ is pseudo $p$-adically closed ($\PpC$) if it is is existentially closed, as an $\Lrg$-structure, in every totally $p$-adic regular extension. 
\end{definition}

\begin{remark}
By Lemma\,13.9 of \cite{HarJar-PpC} this is equivalent to every non-empty absolutely irreducible variety $V$ defined over $F$ having an $F$-rational point, provided that it has a simple rational point in each $p$-adic closure of $F$.
 \end{remark}

\begin{fact}[factPpC]({\cite[Theorem 10.8]{Jar-PpC}})
Let $F$ be a $\PpC$ field and let $\val$ be a $p$-adic valuation on $F$. Then: 
\begin{enumerate}[label=(\roman*)]
\item The $p$-adic closure of $F$ with respect to $\val$ is exactly its Henselianization. In particular all $p$-adic closures of $F$ with respect to $\val$ are $F$-isomorphic.
\item $F$ is dense, for the $\val$-topology, in its $p$-adic closure $\pcl{F}$.
\item $\val(F)$ is a $\Zz$ group.
\item If $\val_1$ and $\val_2$ are distinct $p$-adic valuations on $F$, then $\val_1$ and $\val_2$ are independent, i.e. $\val_1$ and $\val_2$ generate different topologies.
\end{enumerate}
\end{fact} 

\begin{remark}[AproThe]
Let $F$ be a $\PpC$ field and let $\val_1, \ldots, \val_n$ be different $p$-adic valuations on $F$. For all $i\leq n$, let $U_i$ be a non-empty subset of $F^r$ which is open in the topology associated to \(\val_i\). Then by \cite[Theorem 4.1]{PreZie} and \ref{factPpC} we have that $\bigcap_i U_i \not = \emptyset$.
\end{remark}
 
\begin{definition}[defnPpc] 
Let $F$ be a field, $n\geq 1$, and let $\val_1, \ldots, \val_n$ be $n$ distinct $p$-adic valuations on $F$. 
The field $(F, \val_1, \ldots, \val_n)$ is $n$-pseudo $p$-adically closed ($\PpC[n]$) if $F$ is a $\PpC$ field and $\val_1, \ldots, \val_n$ are the only $p$-adic valuations of $F$. If $U_i$ is an open (respectively closed) set for the topology associated to the valuation $\val_i$, we will say that $U_i$ is $i$-open (respectively $i$-closed). 
\end{definition}

\begin{notation}[Mac nota] 
Let $F$ be a field and for \(0<i\leq n\), let \(\val_i\) be a \(p\)-adic valuation on \(F\) and let us fix a $p$-adic closure $\pM[F]{i}$ of $F$ with respect to $\val_i$. Let $\cL^{(i)}= \Lrg \cup \{\Div_i, \Pow[i]{m} : m \in \Zz_{>0},\,1\leq i\leq n\}$. Let $\cL := \bigcup_{i=1}^n \cL^{(i)}$. We interpret \(x \Div_i y\) as \(v_i(x)\leq v_i(z)\) and $\Pow[i]{m}(x)$ as \(\pM[F]{i} \models \exists y\,y^m= x \wedge x \not = 0\).
\end{notation}

Recall that, by \ref{factPpC}, if \(F\) is \(\PpC[n]\), $\pM[F]{i}$ is unique up to isomorphism, so the \(\cL\)-structure of \(F\) does not depend on the choice of \(\pM[F]{i}\). In particular, if \(L\) is a totally \(p\)-adic extension of \(F\), the \(\cL\)-structure induced by \(L\) on \(F\) is that unique \(\cL\)-structure.  Note also that \(\Pow[i]{m}(\pM[F]{i})\) is an \(i\)-clopen subset of \(\inv{\pM[F]{i}}\).

\begin{fact}[simpointp]({\cite[Lemma 3.6]{EfrJar}})
Let $(F,\val_1, \ldots, \val_n)$ be an $\PpC[n]$ field and let $V$ be an
absolutely irreducible variety defined over $F$.  For each $1\leq
i\leq n$, let $q_i \in V(\pM[F]{i})$ be a simple point.  Then $V$
contains an $F$-rational point $q$, arbitrarily $i$-close to $q_i$
for all $i \in \{1, \ldots, n\}$.
\end{fact}

We can now deduce a strengthening of \cite[Lemma 13.9 (a)]{HarJar-PpC}. Note that in the conclusion \(F\) is existentially closed as an $\cL$-structure and not only as an \(\Lrg\)-structure.

\begin{theorem}[PpCExisClosed]
Let $(F, \val_1, \ldots \val_n)$ be $\PpC[n]$ and let $L$ be a regular totally $p$-adic extension of $F$. Then $F$ is existentially closed in $L$ as an $\cL$-structure.
\end{theorem}

\begin{proof}
Let $\exists x \phi$ be an existential \(\cL(F)\)-formula such that $L \models\exists x \phi$. We need to show that $F \models \exists x \phi$. We may assume that \(\phi(x)\) is of the form:

\[x \in V \wedge \bigwedge_{i \leq n} x \in U_i,\]

where $V$ is Zariski closed over $F$ and $U_i$ is an $i$-open quantifier free $\cL_i(F)$-definable subset of \(V\). Let $a \in L$ be any tuple such that $L \models a \in V \wedge \bigwedge_{i \in I} a \in U_i$.
Since $F\leq L$ is a regular extension, so is $F\leq F(a)$. By \cite[Corollary\,10.2.2]{FriJar-FA}, there is an absolutely irreducible variety $W$, defined over $F$ such that $a$ is a generic point of $W$. Note that \(W\subseteq V\).

For every \(i\), let $\pM[L]{i}$ be the \(p\)-adic closure used to define the \(\cL\)-structure of \(L\) and let \(\pM[F]{i}\subsel\pM[L]{i}\) be a \(p\)-adic closure of \(F\). Since $L \subseteq \pM[L]{i}$, \[\pM[L]{i} \models \psi_i(a):= a\text{ is a simple point of } W \wedge  a \in U_i,\] and hence we find \(b_i\in \pM[F]{i}\) such that \(\pM[F]{i}\models\psi_i(b_i)\). By \ref{simpointp} there is $b \in W(F)$ such that $b$ is arbitrarily $i$-close to $b_i$, for all $i$. In particular, we can choose $b \in \bigcap_i U_i$. Then we have $F \models \phi(b)$.
\end{proof}

\begin{definition}
A field $F$ is bounded if for any $n\in\Zz_{>0}$, $F$ has finitely many algebraic extensions of degree $n$.
\end{definition}

Note that, by \cite[Lemma\,6.1]{Mon-NTP2}, any bounded \(\PpC\) field is \(\PpC[n]\) for some \(n\in\Zz_{\geq 0}\).

\begin{remark}
In Theorem 7.1 of \cite{Mon-NTP2},  the first author showed that if $n\geq 2$ and $F$ is a bounded $\PpC[n]$ field, then $\Th[\Lrg]{F}$ is not $\NIP$. An important ingredient of the proof was that the \(p\)-adic valuations of a bounded \(PpC\)-field are definable --- with parameters --- in the language of rings. As a consequence of \ref{PpCExisClosed}, this result can be generalized to arbitrary \(\PpC[n]\) fields. The proof is exactly as in \cite[Theorem 7.1]{Mon-NTP2}. \ref{PpCExisClosed} allows us to work in \(\cL\) without using boundedness to get that it is a definable expansion of \(\Lrg\).
\end{remark}

\subsection{Notations}\label{s:not}

In this section, we describe the language in which we prove elimination of imaginaries for bounded \(\PpC\) fields. Let us first fix some notation for the rest of the paper.

\begin{notation}[main not]
Fix a bounded \(\PpC\) field $F$ and let $\{\val_1, \ldots, \val_n\}$ be its \(p\)-adic valuations. For each $i \in \{1,
  \ldots, n\}$ let $\Val_i$ denote the valuation ring associated with
  $\val_i$ and let $\Val := \bigcap_i\Val_i$. We denote by $\Mid_i$ the
  maximal ideal of $\Val_i$.
  
  Let $F_0\subsel F$ be a countable elementary substructure (in the language of rings).
  Let \(\bcL_i\) be a copy of \(\LG(F_0)\) sharing the field sort \(\K\)
  and constant symbols for the elements of $F_0$. 
  Let \(\Geom_i\) denote the sorts of \(\bcL_i\). For all $m \in
  \Zz_{>0}$, we denote by \(\Latt[i]{m}\) the sort interpreted as
  \(\GL{m}(\K)/\GL{m}(\Val_i)\) and $\Tor_m^i$ the sort interpreted by
  $\GL{m}(\K)/\GL{m,m}(\Val_i)$. Let $\ACVFG_i$ denote the copy
    of $\ACVF$ in $\bcL_i$. Let \(\cL_i\) denote a definable
    enrichment of $\bcL_i$ in which \(\pCFG_i\), the \(\cL_i\)-theory
    of \(p\)-adically closed fields, eliminates quantifiers. Let
  \(\cL := \bigcup_{i\leq n}\cL_i\) and \(T := \Th[\cL]{F}\), where
  the \(\cL_i\)-structure of \(F\) is induced by its \(p\)-adic closure for \(\val_i\).
\end{notation}

\begin{remark} Observe that $\alg{F} = \alg{F_0}F$ and hence,
    for all \(M\models T\), \(\alg{\K(M)} = \alg{F_0}\K(M)\). So
    \(\K(M)\) is bounded. Moreover the extension \(F_0\leq \K(M)\) is
    regular.
\end{remark}

We now check that we have not added any new structure to \(F\):

  \begin{proposition}
 In models of \(T\), every \(\cL\)-definable set is \(\Lrg(F_0)\)-interpretable.
  \end{proposition}

\begin{proof}
  Let us first prove that \(\LP\) is definable in
  \(\Lrg(F_0)\):

\begin{claim}[Lrg to L]
  Let \(M,N\models T\), \(A = \alg{\K(A)}\cap M\subseteq M\),
  \(B = \alg{\K(B)}\cap N\subseteq N\) and \(f:A\to B\) be an
  \(\Lrg(F_0)\)-isomorphism. Then \(f\) is an \(\cL\)-isomorphism.
\end{claim}
  
\begin{proof}
By quantifier elimination in \(\LP\) and the fact that
\(\Div_i\) can be defined without quantifiers using the predicates
\(\Pow[i]{m}\) of \ref{Mac nota}, it suffices to check that these
predicates are preserved by \(f\). 
Let \(\pM{i}\models\pCFG_i\) contain \(M\). 
Then for any \(a\in A\), \(\Pow[i]{m}(a)\) holds if and
only if there exists \(y\in \pM{i}\) such that \(y^m = a\). 
Let \(L\supstr \K(M)\) be the compositum of all degree \(m\) extensions of
\(\K(M)\) inside \(\pM{i}\). Since \(\K(M)\) is bounded, \(L\) is a
finite extension and there exists \(\alpha\in\alg{F_0}\), such that
\(L = \K(M)[\alpha]\). Then, \(\Pow[i]{m}(a)\) holds if and only if there
exists \(y\in \K(M)[\alpha]\) such that \(y^m = a\). Identifying
\(\K(M)[\alpha]\) with \(\K^l(M)\) for some \(l\), we see that the
coordinates of \(y\) are in \(\alg{(F_0a)}\cap \K(M)\subseteq
\K(A)\). It follows that \(P_{m, i}(a)\) if and only if
\(P_{m,i}(f(a))\).
\end{proof}

It follows that every \(\Val_i\) is \(\Lrg(F_0)\)-definable and hence
that the geometric sorts for the valuation \(i\) are
\(\Lrg(F_0)\)-interpretable.

Now, let \(\phi(x)\) be any \(\bcL_i\)-formula. Let \(f\) be the
canonical projection from some cartesian power of \(\K\) to the sort
of \(x\). The formula \(\psi(y) := \phi(f(y))\) is equivalent, in
\(\ACVFG_i\), to a formula in the language \(\Lrg\cup\{\Div_i\}\) and
hence \(\phi(x)\) is equivalent, in \(\ACVFG_i\) to both \(\forall y\,(f(y)
= x\impform \psi(y))\) and \(\exists y\,(f(y) = x\wedge \psi(y))\). It
follows that these two formulas also define the trace of \(\phi\) in
any model of \(\pCFG_i\). So in models of \(\pCFG_i\), every
\(\bcL_i\)-definable set is interpretable in \(\Lrg\), since \(\Div_i\) is definable in \(\Lrg\). The same
argument, but starting with a \(\cL_i\)-formula and looking at its
trace on \(M\) allows us to conclude.
\end{proof}
  
\begin{notation}[notationPrin]
For the rest of the paper, we fix $M\models T$ sufficiently saturated and homogeneous. Let $\aM{i}\models\ACVFG_i$ be the algebraic closure of $M$, along with an extension of \(\val_i\), and let \(\pM{i}\) be the $p$-adic closure of $M$ inside \(\aM{i}\). We denote by $\acl$, $\acla$, $\aclp$, $\dcl$, $\dcla$ and $\dclp$ the model theoretic algebraic and definable closures in $M$, $\aM{i}$ and $\pM{i}$ respectively.

Let $A\subseteq M$ and $a, b\in M$ be tuples. We denote by $a \equivL[A] b$, $a \equiva{A} b$ and $a \equivp{A} b$ the fact that \(a\) and \(b\) have the same type over \(A\) in \(M\), $\aM{i}$ and $\pM{i}$ respectively. We also denote by \(\tp\), \(\tpa\) and \(\tpp\) the type in \(M\), $\aM{i}$ and $\pM{i}$ respectively.
\end{notation}
    
Observe that $a \equiva{A} b$ (respectively $a \equivp{A} b$) if and only if \(a\) and $b$ have the same quantifier free \(\bcL_i\)-type (respectively \(\cL_i\)-type) in \(M\) over \(A\). Note also that \(a \qfequiv[\cL(A)] b\) if and only \(a\equivp{A} b\), for all \(i\).

\subsection{Orthogonality of the geometric sorts}
\label{s:orth geom}

In this section we will prove that the $\Latt[i]{m}$ sorts are orthogonal. We also show that their structure is the pure structure induced by $\pM{i}$.

\begin{lemma}[lift gen]
Let $m \in \Zz_{>0}$, $A\subseteq\K(M)$ and, for all $i\in \{1, \ldots, n\}$ let $s_i\in\Latt[i]{m}(M)$. Assume $M$ is $\card{A}^+$-saturated, then there exists $c\in\bigcap_i s_i(M)$ such that $\trdeg{c/A} = m^2$.
\end{lemma}

\begin{proof}
For all $i$, $\inv{\Val_i}$ is $i$-open in $\K(M)$ and hence $\GL{m}(\Val_i)\subseteq \K^{m^2}$ is $i$-open and so is $s_i$. Therefore, there exists a product of $i$-closed balls $U_i$ included in $S_i$. 
So it suffices to find $c\in\bigcap_i U_i\subseteq \K^{m^2}(M)$ such that $\trdeg{c/A} = m^2$. 
This is easily seen to reduce (by induction on the dimension and translation) to showing that $\bigcap_i \Val_i = \Val$ contains transcendental elements over any small set of parameters and by compactness to showing that $\Val$ is infinite. But this is an immediate consequence of \ref{AproThe}.
\end{proof}

\begin{proposition}[orth geom]
Let $A\subseteq \K(M)$ and, for all $i\leq n$, let $s_i$ and $s_i'\in\Latt[i]{m}(M)$. 
If $s_i\equivp{A} s_i'$, for all \(i\), then $(s_1,\ldots,s_n)\equivL[A](s_1',\ldots,s_n')$.
\end{proposition}

\begin{proof}
By Lemma\,\ref{lem:lift gen}, there exists $c\in\bigcap_i s_i(M)$ such that $\trdeg{c/A} = m^2 = \card{c}$.

\begin{claim}
There exists $c_i\in s_i'(M)$ such that $c_i\equivp{A} c$.
\end{claim}

\begin{proof}
By compactness, we have to show that for all (quantifier free) $\cL_i(A)$-formula $\phi(x)$ such that $\pM{i} \models \phi(c)$, $s_i'\cap \phi(M) \neq \emptyset$. Note that $s_i\cap \phi(\pM{i})$ is an $\cL_i(A)$-definable subset of $(\K(\pM{i}))^{m^2}$ that contains an element of transcendence degree $m^2$ over $A$, so it is $i$-open. As $s_i\equivp{A} s_i'$, $s_i'\cap \phi(\pM{i})$ is also $i$-open and non-empty. By the density of $M$ in $\pM{i}$, this set has a point in $M$.
\end{proof}

By \cite[Lemma\,6.12]{Mon-NTP2}, $\tp(c/A)\cup\{x\in\bigcap_i s_i'\}$ is consistent in $M$. Let $c'$ realize this type. Then there exists an $\cL(A)$-automorphism of $M$ sending $c$ to $c'$ and hence $s_i = c\cdot\GL{m}(\Val_i)$ to $s_i'= c'\cdot\GL{m}(\Val_i)$.
\end{proof}


\begin{corollary}[orth geom def]
Let $A\subseteq \K(M)$ and $X\subseteq\prod_i\Latt[i]{m_i}$ be $\LL(A)$-definable. Then there
exists finitely many quantifier free $\cL_i(A)$-formulas
$\phi_{i, j}(x_i)$ such that $X = \bigcup_j \prod_i\phi_{i, j}(M)$.
\end{corollary}

\begin{proof}
By \ref{orth geom}, the following set is
inconsistent:
\[\{\psi(x_i)\iffform\psi(x_i') \mid 0 < i \leq n\text{ and }\psi\text{ is an }\cL_i(A)\text{-formula}\}\cup\{(x_i)_i\in X\wedge (x_i')_i\nin X\}.\]
By compactness, it follows that there are $k_i$ formulas $\psi_{i, j}(x_i)$ such that
\[M\models\forall x_1\ldots x_n\,(\bigwedge_{0\leq j < k_i}\psi_{i, j}(x_i)\iffform\psi_{i, j}(x_i'))\impform ((x_i)_{i\leq n}\in X\iffform (x_i')_{i\leq n}\in X).\]
For any formula \(\psi\), \(\psi^0\) denotes \(\neg\psi\) and \(\psi^1\) denotes \(\psi\). For every $\epsilon_i : k_i\to 2$, let $\theta_{i,\epsilon_i}(x_i) = \bigwedge_{0\leq j < k_i}\psi(x_i)^{\epsilon_i(j)}$ and for all tuple $\epsilon = (\epsilon_i)_i$, $\theta_\epsilon(x) = \theta_{i,\epsilon_i}(x_i)$. Then for all $\epsilon$, if $\theta_\epsilon(M)\cap X\neq\emptyset$, then $\theta_\epsilon(M)\subseteq X$. Let $E = \{\epsilon\mid\theta_{\epsilon}(M)\cap X\neq\emptyset\}$. Then
\[X = \bigcup_{\epsilon\in E}\theta_\epsilon(M) = \bigcup_{\epsilon\in E}\prod_i\theta_{i,\epsilon_i}(M).\]
This concludes the proof.
\end{proof}

Define \(\Geomim_i\) to be the set of all \(\cL_i\)-sorts but \(\K\).

\begin{corollary}[orth geom gen]
Let \(A\subseteq M\), \(S_i\) be a product of sorts in \(\Geomim_i\) and \(X\subseteq\prod_i S_i\) be an \(\cL(A)\)-definable subset. Then, there exists quantifier free \(\cL_i(\Geom_i(A))\)-definable sets \(X_{i, j}\subseteq S_i\) such that \(X = \bigcup_j\prod_i X_{i, j}(M)\).
\end{corollary}

\begin{proof}
Let \(X\) be defined by \(\phi(s,\zeta, a)\), where \(\zeta_i\in\Geomim_i(A)\) and \(a\in\K(A)\). Let \(Y\) be the set defined by \(\phi(s, y, a)\). It suffices to prove the proposition for \(Y\). Indeed, if \(Y\) is of the correct form, its fiber \(Y_{\zeta} := \{x \mid (x,\zeta)\in Y\}\) is also of the correct form, so we may assume that \(A\subseteq \K(M)\). Because any finite product of  sorts from \(\Geomim_i\) can be encoded, in \(\pCFG_i\), in any $\Latt[i]{m}$ for large enough \(m\), we may also assume that \(S_i = \Latt[i]{m}\) for some fixed \(m\). We can now apply \ref{orth geom def}.
\end{proof}

\subsection{The algebraic closure}\label{s:acl}

Let us now describe the algebraic closure in \(T\).

\begin{proposition}[acl Gi]
Let $A\subseteq M$. Then, for all \(i\), $\Latt[i]{m}(\acl(A))\subseteq\aclp(\Geom_i(A))$.
\end{proposition}

\begin{proof}
By \ref{orth geom gen}, any (finite) \(\cL(A)\)-definable subset of $\Latt[i]{m}$ is quantifier free \(\cL_i(\Geom_i(A))\)-definable. The proposition follows.
\end{proof}

The sorts $\K$ and $\Latt[i]{m}$ are obviously not orthogonal as there are functions with infinite range from $\K$ to $\Latt[i]{m}$. However, there are no functions with infinite range from $\Latt[i]{m}$ to $\K$:

\begin{proposition}[acl K]
Let $A\subseteq M$. Then $\K(\acl(A))\subseteq\alg{\K(A)}$.
\end{proposition}

\begin{proof}
Let $c\in\K(\acl(A))$. As in the proof of \ref{orth geom gen}, we may assume that there exists \(s_i\in\Latt[i]{m}\), for some fixed \(m\), such that \(c\in\acl(\K(A)(s_{i})_{i\leq n})\). By Lemma\,\ref{lem:lift gen}, there exists $e\in\bigcap_i s_i(M)$ such that $\trdeg{e/\K(A)c} = m^2 = \card{e}$. Then, by \cite[Lemma 5.10]{Mon-NTP2}, $c\in\acl(\K(A)e) \subseteq\alg{\K(A)e}$. But $e$ is algebraically independent from $c$ over $A$. It follows that $c\in\alg{\K(A)}$.  
\end{proof}

\begin{corollary}[descr acl]
Let $A\subseteq M$. Then, \[\acl(A) \subseteq \bigcup_i \aclp(\Geom_i(A)).\]
\end{corollary}
\begin{proof}
This is an easy consequence of Propositions\,\ref{prop:acl Gi} and \ref{prop:acl K}.
\end{proof}

\begin{corollary}[fin image]
Let $S$ be one of the sorts in \(\Geomim_i\) and $S'$ be a sort in \(\Geom_j\) for some \(j\neq i\). Then any $\cL(M)$-definable function $S \to S'$ has finite image.
\end{corollary}

\begin{proof}
If $S' = \Latt[j]{m}$, then this follows immediately from \ref{orth geom gen}. If $S' = \K$, it follows from Proposition\,\ref{prop:acl K} and compactness.
\end{proof}

\begin{corollary}[descr acl ACVF]
Let \(A\subseteq \K(M)\). Then \(\acl(A) \subseteq
\bigcup_i\acla(A)\).
\end{corollary}

\begin{proof}
By \ref{descr acl}, \(\acl(A) \subseteq
\bigcup_i\aclp(A)\). Since the relative algebraic
closure of a field inside a \(p\)-adic field is its \(p\)-adic closure
(cf. \cite[Section\,4.(i)]{HruMarRid}), we have \(\aclp(C)
\subseteq \acla(C)\).
\end{proof}

We have proved that hypothesis (\(\star\)) of Section\,\ref{s:approx}
holds of \(\ACVFG_i\) and \(T\). Let us now prove that (\(\dagger\)) also holds:

\begin{lemma}[descent]
Let $\epsilon\in\dcla(M)$, for some \(i\leq n\). Then there exists $\eta\in M$ such that \(\epsilon\) and \(\eta\) are interdefinable in the pair \((\aM{i},M)\).
\end{lemma}

\begin{proof}
We know by \ref{factPpC} that the Henselianization of $\K(M)$ with respect to $\val_i$  is the $p$-adic closure $\K(M_i)$ of $\K(M)$. 
It follows that if $\epsilon\in \K(\dcla(M))$, then $\epsilon\in\K(\pM{i})$. 
As $\alg{\K(M)} = \alg{F_0}\K(M)$, and $\K(\pM{i}) \subseteq \alg{\K(M)}$ there is $b \in \alg{F_0}$ such that $\K(M)[\epsilon]= \K(M)[b]$. Since \(\epsilon\in\pM{i}\), it follows that \(b\in \alg{F_0}\cap \K(M_i) :=  F_1 \subseteq\dcla(\emptyset)\). Let \(\eta\) be the unique tuple in \(M\) such that \(\epsilon = \sum_i \eta_ib^i\).  Then \(\eta\) and \(\epsilon\) are interdefinable in the pair \((\aM{i},M)\).

The case $\epsilon\in \Latt[i]{m}\cup \Tor^i_m$ is tackled as in  
  \cite[(ii) p.\,30]{HruMarRid}. 
  Let us first assume that \(\epsilon\in\Tor^i_m\).
  Find a finite extension \(L\) of \(\K(\pM{i})\), of degree \(r\), in which the lattice coded by
  \(\tau_m(\epsilon)\), denoted $\latt(\epsilon)$, has a basis and the
  coset coded by \(\epsilon\) has a point \(c\). 
 Since \(\alg{\K(\pM{i})} = \alg{F_1}\K(\pM{i})\), we can find \(a\in\alg{F_0}\) such that $\Val(L) = \Val(\K(\pM{i}))[a]$. Let \(f_a :(\K(\pM{i}))^r\to L\) be the map \(r \mapsto \sum_i r_i
  a^i\). Then \(f_a^{-1}(c+\Mid_i\latt(\epsilon)) \subseteq (\K(\pM{i}))^{mr}\) is coded by some
  \(\eta\in \Tor^i_{mr}(\pM{i}) = \Tor^i_{mr}(M)\). If \(a\) and
  \(a'\) have the same \(\bcL_i(F_0)\)-orbit, then they have the same
  \(\bcL_i(F_1)\)-orbit and hence the same
  \(\bcL_i(\pM{i})\)-orbit. In particular, there exists \(\tau\in\aut[\bcL_i](\aM{i})\) fixing \(M_i\) pointwise such that
  \(\tau(a) = a'\). Then \(f_{a}^{-1}(c + \Mid_i\latt(\epsilon)) = f_{\tau(a)}^{-1}(c+\Mid_i\latt(\epsilon)) =
  f_{a'}^{-1}(c+\Mid_i(\latt(\epsilon)))\). It follows that any automorphism of \(\aM{i}\) that stabilizes \(M\) globally fixes \(\epsilon\) if and only if it fixes \(\eta\). Note that we did not use in the proof that \(\aM{i}\) was the algebraic closure of \(M\), so the same argument works in a sufficiently saturated and homogeneous model of the pair \((\aM{i},M)\). So \(\eta\) and \(\epsilon\) are interdefinable. If
  \(\epsilon\in\Latt[i]{m}\), since \(\epsilon\) and
  \(\Mid_i\latt(\epsilon)\) are interdefinable in \(\aM{i}\), we can
  apply the previous case to \(\Mid_i\latt(\epsilon)\).
\end{proof}

We can now apply \ref{inv acl NIP} in our setting. If we
already knew elimination of imaginaries in \(T\), then the following result
would be a rather immediate corollary of the description of the
algebraic closure and the fact that \(\pCF\) and \(\ACVF\) are \(\NIP\). But since we
do not know elimination of imaginaries yet, this is where the encoding
of (germs of) functions happens.

\begin{corollary}[inv acl cor]
Let $i \in \{1, \ldots, n\}$. Let \(A\subseteq \eq{M}\) containing
$\Geom(\acleq(A))$ and \(a \in \K(N)\), where
\(N\supsel M\), be a tuple such that
\(\tpa(a/\aM{i})\in\Inv(\aM{i}/\Geom_i(A))\) and
\(\tpp(\Geom_i(a)/\pM{i})\in\Inv(\pM{i}/\Geom_i(A))\). Then \(\tpa(\Geom_i(\acleq(Aa))/\aM{i})\in\Inv(\aM{i}/\Geom_i(A))\) and \(\tpp(\Geom_i(\acleq(Aa))/\pM{i})\in\Inv(\pM{i}/\Geom_i(A))\).
\end{corollary}

\begin{proof}
We apply \ref{inv acl NIP} twice --- once with \(T_1= \ACVFG_i\) and a second time with \(T_1 = \pCFG_i\).
\end{proof}

\subsection{A local density result}\label{s:dens}
In this section, we prove a weak equivalent for bounded \(\PpC\) fields of the first author's canonical density result for bounded pseudo real closed fields, which plays a key role in proving the elimination of imaginaries for these structures. In \cite{Mon-EIPRC}, the first author proved the following:

\begin{proposition}[globdensPRC][{\cite[Proposition\,4.8]{Mon-EIPRC}}]
Let $(M, <_1, \ldots, <_n)$ be a bounded n-PRC field considered in the language of rings with constants for a countable elementary substructure. Let \(S \subseteq M\) be definable. Then there are a finite set \(S_0 \subseteq S\), \(m \in\Zz_{\geq 0}\) and \(I_j^i\), for \(j < m\), an \(<_i\)-interval in \(M\), such that:
\begin{enumerate}
\item \(S \subseteq S_0\cup \bigcup_j I_j\), where \(I_j := \bigcap_i I_j^i\);
\item for every \(j\), \(S\cap I_j\) is dense in \(I_j\) for the topology generated be all the \(<_i\)-topologies;
\item Every \(I_j\) is \(M\cap\acleq(\code{S})\)-definable.
\end{enumerate}
\end{proposition}

A non-canonical version of that result for \(\PpC\) fields --- where interval are replaced by \(p\)-adic \(1\)-cells, and a weakening of (3) holds --- was proved by the first author, cf. \cite[Theorem\,6.8]{Mon-NTP2}. However, since elimination of imaginaries is all about canonicity, we need something more:

\begin{question}[globdensPpC]
Does the $\PpC[n]$ version of \cite[Theorem 4.4]{Mon-EIPRC} hold?
\end{question}

Instead of answering that question, in this paper, we choose to focus on a related, more tractable, problem by describing types rather than definable sets. We now revert to our standard notation described in \ref{notationPrin}.

\begin{definition}
Let $\aM{i}\models\ACVFG_i$. We denote by $\Balls_i(\aM{i})$ the set
of balls in $\aM{i}$. We consider points to be closed balls of
infinite radius and the whole field to be an open ball of radius
\(-\infty\).

Let \(P \subseteq \Balls(\aM{i})\). We define \(\Gen{P}\), the generic
type of \(\bigcap_{b\in P}b\) as: \[\Gen{P}(x) := \{x\in b\mid b\in
P\} \cup \{x\nin b' \mid \forall b\in P,b'\subset b\}.\]
\end{definition}

By $C$-minimality of $\ACVF$, when it is consistent, \(\Gen{P}\)
generates a complete type. We will not distinguish $\Gen{P}$ from the
complete type it generates. Note that, if \(P\subseteq \Balls_i(A)\),
then \(\Gen{P} \in \Inv(\aM{i}/A)\).

\begin{proposition}[loc dens]
Let $A \subseteq \eq{M}$ containing $\Geom(\acleq(A))$ and
$c\in\K(M)$. For each $i\leq n$, let $P_i =
\{b\in\Balls_i(A)\mid c\in b\}$. Then the partial type:
\[\tp(c/A)\cup\bigcup_i \Gen{P_i}\]
is consistent.
\end{proposition}

\begin{proof}
Assume that this type is not consistent. By compactness, there exists
a $\cL(A)$-formula $\phi(x)$, balls $b_i\in P_i$ and
$b_{i, j}\in\Balls_i(\aM{i})$ with $b_{i, j}\subset b_i$ for all $b_i\in
P_i$, such that $M\models\forall x\,((\phi(x)\wedge\bigwedge_i x\in
b_i)\impform \bigvee_{i, j} x\in b_{i, j})$. Because the valuation
$\val_i$ is $p$-adic, any ball is covered by finitely many subballs
and, because $A$ contains $\Geom(\acleq(A))$, $P_i$ cannot
have a minimal element. It follows that replacing the $b_{i, j}$ by the
smallest ball covering them, we may assume that there is only one
$b_{i, j}$ denoted $b_i'$. For all tuple of balls $b_2',\ldots, b_n'$
(with finite radius) where $b_i'\in\Balls_i$, let
$f_1(b_2',\ldots, b_n')$ be the minimal ball covering
$(\phi(M)\cap\bigcap b_i(M))\sminus \bigcup_{i\geq 2}b_i'(M)$. This
ball exists because, by \ref{factPpC}, $\val_i(M)$ is a pure
$\Zz$-group.

By Corollary \ref{cor:fin image}, $f_1(b_2',\ldots,b_n') \in
A$. Removing (the set encoded by) $f_1(b_2',\ldots, b_n')$ from
$\phi(x)$, we still have $M\models\phi(c)$, but now $M\models\forall
x\,\phi(x)\wedge\bigwedge_i x\in b_i\impform\bigvee_{i\geq 2} x\in b_i'$. By
induction, removing $\eq{\cL}(A)$-definable sets from $\phi(x)$, we
can get to a situation where $M\models\phi(c)$ and
$\phi(M)\cap\bigcap_ib_i(M) = \emptyset$, a contradiction.
\end{proof}

\begin{proposition}[loc dens Qp]
Let $A \subseteq \eq{M}$ containing $\Geom(\acleq(A))$ and 
$c\in\K(M)$. Then we can find types \(p_i\in\Inv(\aM{i}/\Geom_i(A))\)
and \(q_i\in\Inv(\pM{i}/\Geom_i(A))\) such that the partial type:
\[\tp(c/A)\cup\bigcup_i p_i \cup\bigcup_i q_i\]
is consistent.
\end{proposition}

\begin{proof}
By \ref{loc dens}, we can find \(P_i\subseteq\Balls_i(A)\) such that
\(\tp(c/A)\cup\bigcup_i \alpha_{P_i}\) is consistent. If any of the
\(\alpha_{P_i}\) is a realized type, then \(c\in \K(A)\) and we are
done.

We may assume that \(c\models\alpha_{P_i}\). 
Pick any \(d \in\bigcap_{b\in P_i}b(\pM{i})\) and let \(f_m\in \alg{F_0}\cap M\) be
such that \(f_m(c-d) \in \K(\pM{i})^m\). Then for any \(d'
\in\bigcap_{b\in P_i}b(\pM{i})\), \(f_m(c-d') \in \K(\pM{i})^m\) since
\(d\) and \(d'\) are much closer to each other than to \(c\). Note
also that for any \(e \nin \bigcap_{b\in P_i}b(\pM{i})\),
\(f_m(c-e)\in \K(\pM{i})^m\) if and only if \(f_m(d-e)\in
\K(\pM{i})^m\), a value that does not depend on the choice of \(c\).

So \(\restr{\alpha_{P_i}(x)}{\pM{i}}\cup\{f_m(x-d)\in
\K(\pM{i})^m \mid m\in\Zz_{>0}\}\) generates \(q_i(x) := \tpp(c/M_i)\), which is
therefore in \(\Inv(\pM{i}/\Geom_{i}(A))\).
\end{proof}

\begin{remark}
In fact, \ref{loc dens Qp} would follow from a positive answer to \ref{globdensPpC}. The converse is far from clear.

Similarly, the pseudo real closed version of \ref{loc dens Qp} follows from \ref{globdensPRC}.
\end{remark}

The higher dimension version of \ref{loc dens Qp} follows formally by
induction from \ref{inv acl cor}:

\begin{theorem}[loc dens higher]
Let $A \subseteq \eq{M}$ containing $\Geom(\acleq(A))$ and
$c\in\K^m(M)$.  Then we can find types
\(p_i\in\Inv(\aM{i}/\Geom_i(A))\) and
\(q_i\in\Inv(\pM{i}/\Geom_i(A))\) such that the partial type:
\[\tp(c/A)\cup\bigcup_i p_i \cup\bigcup_i q_i\]
is consistent.
\end{theorem}

\begin{proof}
We proceed by induction on \(m\).  Assume that we have
\(p_i(x)\in\Inv(\aM{i}/\Geom_i(A))\),
\(q_i\in\Inv(\pM{i}/\Geom_i(A))\) and \(a\models\bigcup_i
p_i\cup\bigcup_i q_i\). Pick some \(c\in\K(M)\). We want to find types
\(r_i(x, y)\in\Inv(\aM{i}/\Geom_i(A))\) and
\(s_i(x, y)\in\Inv(\pM{i}/\Geom_i(A))\) such that
\(\tp(ac/A)\cup\bigcup_i r_i\cup\bigcup_i s_i\) are consistant.

Let \(E_i = \Geom_i(\acleq(Aa))\). By \ref{inv acl cor}, we find
\(p'_i\in\Inv(\aM{i}/\Geom_i(A))\) and
\(q'_i\in\Inv(\pM{i}/\Geom_i(A))\) such that \(E_i\models p'_i\cup
q'_i\). Let \(E := A\cup\bigcup_iE_i\), then
\(\Geom(\acleq(E))\subseteq E\). Let \(N \supsel M\)
containing \(M\cup E\), \(N_i\) and \(\alg{N_i}\) as in
\ref{notationPrin}. Applying \ref{loc dens}, we can find
\(p''_i\in\Inv(\alg{N_i}/E_i)\), \(q''_i\in\Inv(\pM{i}/E_i)\) and
\(c^*\models\tp(c/E)\cup\bigcup_i p''_i\cup q_i''\). Let \(r_i(x, y) :=
\{\phi(x, y)\mid \phi(x,a)\in p''_i\}\) and \(s_i(x, y) :=
\{\phi(x, y)\mid \phi(x, a)\in q''_i\}\). It is easy to check that these
two types have the required properties.
\end{proof}

\begin{remark}
 Note that, so far, we have not used \cite[Theorem\,2.6]{HruMarRid} ---  elimination of imaginaries in \(p\)-adically closed fields. Moreover, if \(M\) is \(p\)-adically closed, then there is just one
  \(p\)-adic valuation on \(M\) and the type \(q_1\) that we
  constructed is a complete type. We have just reproved a strong version of the invariant extension property (cf. \cite[Corollary\,4.7]{HruMarRid}), of which, by \ref{QE crit}, weak elimination of imaginaries follows. The main difference between the two proofs is that the proof presented here focuses from the start on finding invariant extensions. The arguments in this paper could also easily be carried out in finite extensions of \(p\)-adically closed fields.
\end{remark}

\subsection{Amalgamation over geometric points}\label{s:amalg}

The goal of this section is to improve Montenegro's amalgamation
result \cite[Theorem 3.21]{Mon-NTP2} to allow amalgamation over bases
of geometric points.

\begin{fact}
Let \(E\substr A\substr C\) be field extensions.
\begin{enumerate}[label=(\roman*),ref=\thetheorem.(\roman*)]
\item\label{fct:full reg} If \(E\substr C\) is regular, then
  \(\alg{E}A\cap C = EA\).
\item\label{fct:top reg} If \(A\substr C\) is regular, then
  \(\alg{E}C\cap \alg{A} = \alg{E}A\)
\end{enumerate}
\end{fact}

\begin{proof}
\begin{enumerate}[label=(\roman*)]
\item Since \(C\) is linearly disjoint from \(\alg{E}\) over \(E\),
  then \(C\) is linearly disjoint from \(\alg{E}A\) over \(A\).
\item Since \(\alg{A}\) is linearly disjoint from \(C\) over \(A\), and \(A \substr \alg{E}A \substr \alg{A}\), we also have that \(\alg{A}\) is linearly disjoint from \(\alg{E}AC = \alg{E}C\) over \(\alg{E}A\).\qedhere
\end{enumerate}
\end{proof}

\begin{lemma}[desc alg]
Let \(A\subseteq \K(M)\). If \(\alg{A}\cap \K(M)\subseteq A\), then \(\alg{A} = \alg{F_0}A\).
\end{lemma}

\begin{proof}
We have \(\alg{A}\subseteq \alg{\K(M)} = \alg{F_0}\K(M)\). So \(\alg{A}\subseteq \alg{F_0}\K(M) \cap \alg{A} = \alg{F_0}A\). The equality follows from Fact\,\ref{fct:top reg} and the fact that \(A\substr\K(M)\) is regular.
\end{proof}

\begin{lemma}[alg disj]
Let \(A \subseteq M\subsel N\models T\) and \(\K(A)\subseteq C\subseteq\K(N)\). Assume that \(\K(\dcl(A))\cap M\subseteq A\) and that, for some \(i\leq n\), \(\qftp_{\cL_i}(C/M)\) is \(\aut[\cL](M/A)\)-invariant. Then \(C\) and and \(\K(M)\) are algebraically disjoint over \(\K(A)\).
\end{lemma}

\begin{proof}
Let \(m\) be any finite tuple in \(C\) and \(V\) be its algebraic locus over \(\K(M)\). We have to show that \(V\) is defined over \(\K(B)\). Pick any \(\sigma\in\aut[\cL](M/A)\). Since \(\qftp_{\cL_i}(C/M)\) is \(\aut[\cL](M/A)\)-invariant, we also have \(m\in\sigma_i(V) = \sigma(V)\) and hence \(V\subseteq \sigma(V)\). It follows, using \(\sigma^{-1}\), that \(V = \sigma(V)\) and, by elimination of imaginaries in \(\ACF\), \(V\) is defined over \(\K(\dcl(B)) = \alg{\K(B)}\cap M = \K(B)\).
\end{proof}

Let us first generalize the characterization of types in bounded pseudo \(p\)-adically closed fields (cf. \cite[Proposition\,10.4]{Jar-PpC} and \cite[\S 6.4]{Mon-NTP2}) to allow geometric parameters:

\begin{lemma}[forth]
Let \(M\), \(N\models T\), \(A\subseteq M\), \(B\subseteq N\) and \(f : A\to B\) be an \(\cL\)-isomorphism. Assume that \(\alg{\K(A)}\cap M \subseteq A\) and \(\alg{\K(B)}\cap N \subseteq B\) and \(N\) is \(\card{M}^+\)-saturated. Then there exists an \(\cL\)-embedding \(g:M\to N\) such that \(\alg{\K(g(M))}\cap N\subseteq g(M)\).
\end{lemma}

\begin{proof}
Let \(\fU_i\models\pCFG_i\) be a sufficiently saturated and homogeneous model that contains \(\Geom_i(M)\cup\Geom_i(N)\). Let \(A_i := \aclp(A)\), \(B_i := \aclp(B)\) and \(f_i : A_i \to B_i\) be an \(\cL_i\)-isomorphism extending \(f\). 
Also, for each \(i\), let \(p_i\in\Inv(\fU_i/B_i)\) extend \(\push{(f_i)}{\tpp(\K(M)/A_i)}= \{\phi(x, f_i(a))\mid \phi(x, a)\in\tpp(\K(M)/A_i)\}\). Let \(F_i \models \restr{p_i}{NA_i}\). Note that the fields \(F_i\) are all isomorphic as fields (over \(\K(N)\K(B)\)) so we may identify them. Let \(M^\star\) be the \(\cL\)-structure whose underlying field is \(F_i\) and whose \(\cL_i\)-structure is induced by \(\fU_i\). Note that there is an \(\cL\)-isomorphism between \(M\) and \(M^\star\) extending \(f\).

Note that since any automorphism of \(M\) extends to \(\fU_i\), \(\qftp_{\cL_i}(M^\star/N)\) is \(\aut[\cL](N/B)\)-invariant. It follows, by \ref{alg disj}, \(\K(M^\star)\) is algebraically independent from \(\K(N)\) over \(\K(B)\). Also, the extension \(\K(A)\substr\K(M)\) is regular, hence so are the extensions \(\K(B)\substr\K(M^\star)\) and \(\K(N)\substr \K(M^\star)\K(N)\). Since, \(\K(M^\star)\K(N)\substr\fU_i\), for all $i \leq n$, this extension is totally \(p\)-adic. Let \(N_0\subsel N\) be a small model containing \(B\). As \(N\) is pseudo \(p\)-adically closed, there exists \(M^\diamond \substr N\) which is \(\Lrg(N_0)\)-isomorphic to \(M^\star\). Since \(\cL\) is a definable enrichment of \(\Lrg(F_0)\), it follows that \(M^\diamond\) is \(\cL(B)\)-isomorphic to \(M^\star\). Composing with the \(\cL\)-isomorphism between \(M\) and \(M^\star\), we get an \(\cL\)-isomorphism \(g: M \to M^\diamond\) extending \(f\).

\begin{claim}
\(\alg{\K(M^\diamond)}\cap N = \K(M^\diamond)\)
\end{claim}

\begin{proof}
Note that \(M^\diamond\) is \(\cL\)-isomorphic to \(M\) and hence \(M^\diamond\models T\). It follows that \(\alg{\K(M^\diamond)} = \alg{F_0}\K(M^\diamond)\). Since \(F_0\substr N\) is regular, by Fact\,\ref{fct:full reg}, \(\alg{\K(M^\diamond)}\cap N = \alg{F_0}\K(M^\diamond) \cap N = \K(M^\diamond)\).
\end{proof}

This concludes the proof.
\end{proof}

\begin{proposition}[type PpC]
Let \(M\), \(N\models T\), \(A\subseteq M\), \(B\subseteq N\) and \(f : A\to B\) be an \(\cL\)-isomorphism. Assume that \(\alg{\K(A)}\cap M \subseteq A\) and \(\alg{\K(B)}\cap N \subseteq B\). Then \(f\) is elementary.
\end{proposition}

\begin{proof}
Assume \(M\) and \(N\) are sufficiently saturated. We proved in \ref{forth} that the set of \(\cL\)-isomorphisms between (small) \(A \subseteq M\) and \(B\subseteq N\) such that \(\alg{\K(A)}\cap M \subseteq A\) and  \(\alg{\K(B)}\cap N \subseteq B\) has the back-and-forth (if it is non-empty). It follows that any such isomorphism is elementary.
\end{proof}

Using the results of Section\,\ref{s:orth geom} and and the fact that the \(\cL\)-structure on the sort \(\K\) in \(T\) is a definable expansion of the ring language, we can improve this last result:

\begin{corollary}[desc types]
Let \(F_0\subseteq E\subseteq M\) and \(a\in M\) be a tuple, then \[\qftp_{\cL}(a/E) \cup\qftp_{\Lrg(F_0)}(\alg{\K(E)\K(a)}\cap M/\K(E))\vdash \tp(a/E).\]
\end{corollary}

\begin{proof}
It suffices to prove that if we have \(A, B\subseteq M\) containing \(F_0\),\(f:A\to B\) an \(\cL\)-isomorphism  and \(g:\alg{\K(A)}\cap M\to\alg{\K(B)}\cap M\) an \(\Lrg(F_0)\)-isomorphism such that \(\restr{f}{\K(A)} = \restr{g}{\K(A)}\), then \(f\) is an elementary \(\cL\)-isomorphism. By \ref{Lrg to L}, \(g\) is an \(\cL\)-isomorphism. By \ref{type PpC}, \(g\) is an elementary \(\cL\)-isomorphism which extends to \(\sigma\in\aut[\cL](M)\). Let \(C := \sigma(A)\). It suffices to prove that \(C\equivL B\). We know that \(\K(C) = \K(B)\) so it suffices to prove that \(\Geomim(C)\equivL[\K(B)] \Geomim(B)\), where \(\Geomim\) is the set of all \(\cL\)-sorts except for \(\K\). By hypothesis, \(C\qfequiv[\cL] A \qfequiv[\cL] B\) so \(\Geomim(C)\qfequiv[\cL(\K(B))]\Geomim(B)\). But, by \ref{orth geom}, this is equivalent to \(\Geomim(C)\equivL[\K(B)]\Geomim(B)\).
\end{proof}

\begin{lemma}[p-adic amalg]
Let \(L_1\), \(L_2\models\pCF_{\forall}\) and \(k_0\) be a common $\LP$-substructure of \(L_1\) and \(L_2\) such that \(L_1\) and \(L_2\) are linearly disjoint over \(k_0\). Then \(L_1L_2\) can be made into an \(\LP\)-structure extending that \(L_1\) and \(L_2\) such that \(L_1L_2\models\pCF_{\forall}\).
\end{lemma}

\begin{proof}
Since \(\pCF\) eliminates quantifiers in \(\LP\), there exists
\(\fU\models pCF\) containing both \(L_1\) and \(L_2\) as
\(\LP\)-substructures. By induction, it suffices to consider the case
where \(L_1 = k_0(a)\). If \(a\) is algebraic over \(k_0\), then the
minimal polynomial of \(a\) over \(k_0\) and \(L_2\) coincide. So
\(k_0(a)\) and \(L_2\) are linearly independent over \(k_0\) (as
subfields of \(\fU\)) and we are done. If \(a\) is transcendental over \(k_0\), let \(c\in \fU\sminus
\alg{L_2}\) realize \(\tp[\LP](a/k_0)\). Then \(k_0(c)\) is linearly
independent from \(L_2\) over \(k_0\) and we are also done.
\end{proof}

\begin{remark}
The above proof is not really about \(p\)-adically closed fields. It
holds of any theory of (enriched) fields that eliminates quantifiers
and is algebraically bounded.
\end{remark}

\begin{theorem}[indep]
Let \(M\models T\), \(E \subseteq M\) and \(a_1\),\(a_2\), \(c_1\), \(c_2\), \(c\in \K(M)\) be such that \(\alg{\K(E)} \cap M \subseteq E\), \(\alg{\K(E)(a_1)}\cap\alg{\K(E)(a_2)} = \alg{\K(E)}\), \(c\qfindep_{E}a_1a_2\), \(c_1 \equivL[E] c_2\), \(c\qfequiv[\cL(Ea_1)] c_1\) and \(c\qfequiv[\cL(Ea_2)] c_2\). Then \(\tp(c_1/Ea_1)\cup\tp(c_2/Ea_2)\cup\qftp_{\cL}(c/Ea_1a_2)\) is satisfiable.
\end{theorem}

This is very similar to a result proved by the first author \cite[Theorem\,6.13]{Mon-NTP2}. The main difference is that in \ref{indep}, \(E\) is allowed to contain geometric points.

\begin{proof}
Let \(p\) be a quantifier free \(\aut[\cL](M/E)\)-invariant type extending the quantifier free type of \(c\) over \(E a_1 a_2\). Choosing \(c\models p\) in some \(N\supsel M\), we may assume that \(c\qfindep_{E} M\).  By \ref{alg disj}, \(\K(E)(c)\) is algebraically disjoint from \(\K(M)\) over \(\K(E)\). Let \(A_j := \alg{\K(E)a_j}\cap M\), \(C_j := \alg{\K(E)c_j}\cap M\), \(C := \alg{\K(E)c}\cap M\), \(F = \alg{A_1A_2}\cap M\) and \(B_j := \alg{A_jC_j}\cap M\). Since \(c\qfequiv[\cL(Ea_j)] c_j\) and \(\aclp = \dclp\), we have that \(C\qfequiv[\cL(A_j)] C_j\). So there exists an \(\cL(A_j)\)-isomorphism \(\phi_j:A_jC_j \to A_jC\) sending \(C_j\) to \(C\). We can extend this isomorphism to an \(\Lrg(F_0)\)-isomorphism \(\alg{\phi}_j:\alg{A_jC_j} \to \alg{A_jC}\). Let \(D_j := \alg{\phi}_j(B_j)\). A picture of the involved fields might help:

\[\xymatrix{
B_1\ar[r]^{\alg{\phi}_1}&D_1&F&D_2&B_2\ar[l]_{\alg{\phi}_2}\\
A_1C_1\ar@{-}[u]\ar[r]^{\phi_1}&A_1C\ar@{-}[u]&A_1A_2\ar@{-}[u]&A_2C\ar@{-}[u]&A_2C_2\ar@{-}[u]\ar[l]_{\phi_2}\\
C_1\ar@{-}[u]&A_1\ar@{-}[u]\ar@{-}[ul]\ar@{-}[ur]|!{[u];[r]}\hole&C\ar@{-}[ul]\ar@{-}[ur]&A_2\ar@{-}[ul]|!{[u];[l]}\hole\ar@{-}[u]\ar@{-}[ur]&C_2\ar@{-}[u]\\
\K(E)(c_1)\ar@{-}[u]&\K(E)(a_1)\ar@{-}[u]&\K(E)(c)\ar@{-}[u]&\K(E)(a_2)\ar@{-}[u]&\K(E)(c_2)\ar@{-}[u]\\
&&\K(E)\ar@{-}[ull]\ar@{-}[ul]\ar@{-}[u]\ar@{-}[ur]\ar@{-}[urr]
}\]

By \cite[Lemma 2.5.(2)]{Cha-OmPAC},
\[\alg{CA_1}\cap\alg{CA_2}\alg{A_1A_2} = \alg{\alg{C}(\alg{A_1}\cap\alg{A_2})}\alg{A_1} = \alg{C}\alg{A_1} = \alg{\K(E)}CA_1.\] The last equality follows from \ref{desc alg}. Since \(\K(E)\substr C_2F\substr \K(M)\) is regular, it follows \(\K(E)\substr D_2F\) is regular and hence that \(D_2F\) is linearly disjoint from \(\alg{\K(E)}CA_1\) over \(CA_1\). In conclusion, we have that \(\alg{CA_1}\cap D_2F = \alg{\K(E)}CA_1\cap D_2F = CA_1\), i.e. the extension \(CA_1\substr D_2F\) is regular. By a symmetric argument, the extension \(CA_2\substr D_1F\) is also regular. Since \(D_2F\) is linearly disjoint from \(D_1\substr \alg{CA_1}\) over \(CA_1\) and the extension \(CA_1\substr D_2F\) is regular, then so is the extension \(D_2F\substr D_1D_2F\). Since \(\K(E)\substr F\) is regular, it follows that \(D_1D_2F\) is linearly disjoint from \(\alg{F} = \alg{\K(E)}F\) over \(F\), i.e. the extension \(F\substr D_1D_2F\) is regular.

Note that \(CF \subseteq N\) can be made into a model of \(\bigcup_i\pCF_{i,\forall}\). Also, \(B_j\) can be made into a model of \(\bigcup_i\pCF_{i,\forall}\) and hence so does \(D_j\) by transfert. Note that since \(\phi_j\) is an \(\cL\)-isomorphism, the \(\cL_i\)-structures induced by \(D_j\) and \(CF\) on \(CA_j\) coincide. Note also, that since \(CA_j\substr D_lF\), where \(l\neq j\) is regular, \(D_j\) and \(CF\) are linearly disjoint over \(CA_j\). By \ref{p-adic amalg}, it follows that \(D_jF\) can be made into a model of \(\pCF_{i,\forall}\) whose \(\cL_i\)-structure extends that of \(CF\). Since \(CA_1\substr D_2F\) is regular, we also have that \(D_1F\) and \(D_2F\) are linearly disjoint over \(CF\). By \ref{p-adic amalg} again, \(D_1D_2F\) can be made into a model of \(\pCF_{i,\forall}\) whose \(\cL_i\)-structure extends that of \(F\).

Recall that \(\K(E)(c)\) is algebraically disjoint from \(\K(M)\) over \(\K(E)\). It follows that \(CF\) is algebraically disjoint from \(\K(M)\) over \(F\) and thus that \(D_1D_2F\subseteq \alg{CF}\) is algebraically disjoint from \(\K(M)\) over \(F\). Since \(F\substr \K(M)\) is regular, \(D_1D_2F\subseteq \alg{CF}\) and \(\K(M)\) are in fact linearly disjoint over \(F\) and, since \(F\substr D_1D_2F\) is also regular, \(M\substr D_1D_2\K(M)\) is regular. Moreover, by \ref{p-adic amalg}, \(D_1D_2M\) can be made into a model of \(\pCF_{i,\forall}\) for all \(i\), i.e. the extension if totally \(p\)-adic. Since \(M\models \PpC[n]\), by \ref{PpCExisClosed}, there exists \(M^\star\supsel M\) containing \(D_1D_2F\) as an \(\cL\)-substructure.

Since \(\alg{B_j} = \alg{\K(E)}B_j\), we also have \(\alg{D_j} = \alg{\K(E)}B_j\). As \(\K(E)\) is regular in \(\K(M^\star)\), it then follows that \(\alg{D_j}\cap M^\star = D_j\). Note that, since the \(\cL\)-structure on \(M(c)\) coincides with the one we built on \(D_1D_2F\) and hence with the one in \(N(c)\), \(\qftp_{\cL}^{M^\star}(c/Ea_1a_2) = \qftp_{\cL}^{N}(c/Ea_1a_2)\) has not changed and we still have that \(c \qfequiv[\cL(Ea_j)]c_j\). By construction, we also have \(D_j \qfequiv[\cL(A_j)] B_j\). It follows, from \ref{desc types}, that \(c\equivL[Ea_j] c_j\).
\end{proof}

\subsection{Finite sets}\label{s:finite}

Our goal, in this section, is to prove that finite sets are coded. We first prove that finite sets are coded in the algebraic closure of a \(\PpC[n]\) field equipped with extensions of the geometric language for each valuation and then we conclude with \ref{descent}. The proof technique is inspired by Johnson's account of elimination of imaginaries in algebraically closed valued fields \cite[\S 6.2]{Joh-EIACVF}.

In what follows, let \(F\) be an algebraically closed field and \(\val_i\), \(1\leq i \leq n\), be \(n\) independent valuations on \(F\). Recall \ref{main not}. Let \(\Val_i\) also denote the valuation ring for \(\val_i\) and \(\Val = \bigcap_i\Val_i\). Let \(\qfLG := \bigcup_i \bcL_i\). The field \(F\) can naturally be made into an \(\qfLG\)-structure.

A quantifier free \(\qfLG\)-type \(p(x)\) over \(F\) is said to be \emph{definable} if for each quantifier free \(\bcL_i\)-formula \(\phi(x, y)\), there exists a quantifier free \(\bcL_i\)-formula \(\theta(y)\) such that \(\phi(x, a)\in p\) if and only if \(\models\theta(a)\). If \(p(x)\) and \(q\) are quantifier free \(\qfLG\)-types over \(F\), then we define \(p\tensor q\) to be the quantifier free definable \(\qfLG\)-type of tuples \(ab\) where \(a\models p\) and \(b\models \restr{q}{Fa}\). It is also a quantifier free definable \(\qfLG\)-type. The type \(p\) is said to be \emph{symmetric} if for any quantifier free definable \(\qfLG\)-type \(q\), \(p\tensor q = q\tensor p\), i.e. if \(a\models p\) and \(b\models\restr{q}{Fb}\), then \(a\models\restr{p}{Fb}\). This happens if and only if, for all \(i\), \(\restr{p}{\bcL_i}\) is a symmetric definable type in \(\ACVFG_i\).

\begin{lemma}[lift st dom]
Pick any \(s\in\Latt{m}(F) := \GL{m}(F)/\GL{m}(\Val)\). There exists a
quantifier free symmetric \(s\)-definable \(\qfLG\)-type \(q_s\) such that
\(q_s(x)\vdash x\in s\).
\end{lemma}

\begin{proof}
Let \(k_i\) be the residue field for the valuation \(\val_i\) and \(p_i\) be the generic type of \(\GL{m}(k_i)\). Let \(q_i\) be the type of matrices in \(\GL{m}(\Val)\) whose reduct modulo \(\Mid\) realizes \(p_i\). By \cite[Lemma\,6.7 and Proposition\,2.10,(1)]{HruRid-Meta}, \(q_i\) is a symmetric definable type. Note that, by independence of the valuations, \(q := \bigcup_i q_i\) is consistent and also that, since \(q_i\) is invariant by multiplication in \(\GL{m}(\Val_i)\), \(q\) is invariant by multiplication in \(\GL{m}(\Val)\).

Pick any \(M\in s(F)\). The quantifier free definable \(\qfLG\)-type \(q_s\) of elements \(MN\) for some \(N\models q\) does not depend on the choice of \(M\) and it is quantifier free symmetric \(s\)-definable.
\end{proof}

\begin{lemma}[code fin AC]
Let \(C\) be a finite subset of \(\Latt{m}(F)\times F^l\) for some \(m\) and \(l\). Then \(C\) is coded in \(\Geom := \bigcup_i \Geom_i\).
\end{lemma}

\begin{proof}
By \ref{lift st dom}, for all \(c\in C\), there exists a quantifier free symmetric \(c\)-definable type \(p_c\) concentrating on \(c\), i.e. if \(c = c_1c_2\) with \(c_2\in \K^l\), then \(p_c(x_1,x_2)\models x_1 \in c_1\wedge x_2 = c_2\). Let \(s\) be the definable map sending a finite subset of \(\K^{\card{x_1}+\card{x_2}}\) of size \(\card{C}\) to its code. Note that we can choose the image of \(s\) to be a subset of some \(\K^{m}\). Let \(p_C = \push{s}{(\bigotimes_{c\in C} p_c)}\). Note that since each \(p_c\) is symmetric, the type \(p_C\) does not depend on a choice of enumeration for \(C\). So \(p_C\) is a \(\code{C}\)-definable quantifier free type and \(C\) is \(\bcL(\code{p_C})\)-definable. By elimination of imaginaries in \(\ACVF_i\), each \(\restr{p_C}{\bcL_i}\) has a canonical basis in \(\Geom_i\) and since \(p_C = \bigcup_i \restr{p_C}{\bcL_i}\), \(p_C\), and hence \(C\), is \(\Geom(\code{C})\)-definable.
\end{proof}

\begin{corollary}[code fin]
Let \(M\models T\), and \(C\) be a finite subset of \(\Geom^k(M)\). Then \(C\) is coded in \(\Geom\).
\end{corollary}

\begin{proof}
Any product of sorts from \(\Geomim_i\) can be \(\cL_i\)-definably embedded in \(\Latt[i]{m}\) for large enough \(m\). Also, there is an \(\cL\)-definable injection \(\Latt[i]{m}(\K)\to\Latt{m}(\K)\) sending a coset of \(\GL{m}(\Val_i)\) to its intersection with \(\bigcap_{j\neq i}\GL{m}(\Val_j)\). So we may assume that \(C\) is a subset of \(\Latt{m}\times \K^l\) for some \(m\) and \(l\). By \ref{code fin AC}, \(C\) is coded by some \(\epsilon \in \alg{M}\). Note that since \(C\) is \(\bcL(M)\)-definable, \(\Geom_i(\epsilon)\subseteq\dcla(M)\). Thus, by \ref{descent}, we can find \(\eta_i\in M\) such that \(\Geom_i(\epsilon)\) and \(\eta_i\) are interdefinable in the pair \((\aM{i},M)\). It follows that \(\bigcup_i\eta_i\) is a code for \(C\) in \(M\).
\end{proof}

\subsection{The elimination of imaginaries}\label{s:main}

We can now prove our main result regarding bounded \(\PpC\) fields:

\begin{theorem}[EIPpC]
  The theory $T$ eliminates imaginaries.
\end{theorem}

\begin{proof}
  We apply \ref{crit wEI} and \ref{indep suff} to obtain weak
  elimination of imaginaries. Hypothesis (i) is proved in \ref{loc
    dens higher} and Hypothesis (ii') is proved in \ref{indep}. Since,
  by \ref{code fin}, finite sets are coded, we obtain elimination of
  imaginaries.
\end{proof}

\begin{remark}
As noted in \ref{rem G}.\ref{Tn non nec}, for all \(m\), \(\Tor^{i}_{m}\) is \(\cL_i\)-definably embedded in \(\Latt[i]{m}\), so the \(\Tor^{i}_{m}\) are not necessary to eliminate imaginaries. Also the map sending a coset of \(\GL{m}(\Val)\) in \(\GL{m}(\K)\) to the tuple of \(\GL{m}(\Val_i)\) cosets is an \(\cL\)-definable bijection. Its inverse is given by taking the intersection of the \(\GL{m}(\Val_i)\) cosets. So \(T\) also eliminates imaginaries in the language with a sort for \(\K\) and, for all \(m\in\Zz_{>0}\), a sort \(\Latt{m}\) for \(\GL{m}(\K)/\GL{m}(\Val)\).
\end{remark}
\subsubsection{A pseudo real digression}\label{S:PRC}

Recall that a pseudo real closed field is a field which is existentially closed in any regular extension to which every order extends. Elimination of imaginaries for bounded pseudo real closed fields, in the language of rings with constants for an elementary subfield, was proved in \cite{Mon-EIPRC}. However, there seems to be an error in the final arguments. 

Recall the set up of \cite[Lemma\,4.9]{Mon-EIPRC}. Let M be a sufficiently saturated pseudo real closed field, in the language of rings with constants for a countable elementary substructure. Pick \(e \in\eq{M}\), and \(a\in M\) a tuple such that \(e \in\dcleq(a)\). Let \(E = \acl(E)\supseteq\acleq(e)\cap M\) and \(\overline{E} = \acleq(Ee)\cap M\). Then by the arguments of  \cite[Claim 1 of Proposition\,3.1]{Hru-PAC}, we find \(b_1, b_2\in M\) such that \(\tp(b_1/Ee) = \tp(b_2/Ee) = \tp(a / Ee)\) and \(\acl(\overline{E}b_1)\cap \acl(\overline{E}b_2) = \overline{E}\). 

However, unless we already know weak elimination of imaginaries in bounded pseudo real closed fields, we could have \(E \subset \overline{E}\). Since \(\acl(Eb_1)\cap\acl(E b_2) = \acl(\overline{E} b_1)\cap\acl(\overline{E}b_2) = \overline{E}\), \(b_1\) and \(b_2\) would not be algebraically independent over \(E\) which is crucial later in the proof.

The correct version of \cite[Lemma\,4.9]{Mon-EIPRC} is therefore:

\begin{lemma}[corr lem]
Let \(e\in\eq{M}\), \(f\) a \(\emptyset\)-definable map and \(a\in M\) a tuple such that \(f(a) = e\). Let \(E\subseteq M\) be such that \(\acleq(Ee)\cap M \subseteq E\). Then there exists tuples \(b_1, b_2 \in M\) which are algebraically independent over \(E\), such that \(\tp(b_1/Ee) = \tp(b_2/Ee) = \tp(a/Ee)\).
\end{lemma}

But in \cite[Claim\,1 of Theorem\,4.11]{Mon-EIPRC}, given \(a' \subseteq a\) such that \(\trdeg{E(a)/E(a')} = 1\), if we replace \(E\) by \(\acl(E(a'))\), the stronger hypothesis of \ref{corr lem} might not hold anymore. This is easy to fix using the technology developed in this paper:

\begin{lemma}[loc dens PRC]
Let \(M\) be a bounded pseudo real closed field and let \(X\) be \(M\)-definable. There exists \(p\in\Def(\alg{M}/\acleq(\code{X})\cap M)\) consistent with \(X\).
\end{lemma}

\begin{proof}
Take \(p\) to be the generic type of any irreducible component of the Zariski closure of \(X(M)\). By elimination of imaginaries in algebraically closed fields (and Galois theory), \(p\) is \(\acleq(\code{X})\cap M\)-definable.
\end{proof}

Now in the proof of \cite[Theorem\,4.8]{Mon-EIPRC}, given \(e\in \eq{M}\) in the range of some \(\emptyset\)-definable function \(f\), one can, using \ref{loc dens PRC}, first find \(p\in\Def(\alg{M}/E)\) consistent with \(f^{-1}(e)\) and then choose \(a\models p\). It then follows from \ref{gen stab acl} that for any \(a'\subseteq a\), \(E' := \acleq(e a')\cap M = \acl(Ea')\)  and hence that \(\acleq(E'e)\cap M \subseteq E'\). We can thus safely apply \ref{corr lem} to \(E'\).\medskip

Another, somewhat overkill, approach would be to adapt the general outline of the proof presented in this paper. The only result on bounded pseudo \(p\)-adically closed fields which is proved in this paper and whose pseudo real closed equivalent is not already proved in \cite{Mon-NTP2} is \ref{loc dens Qp}. But the bounded pseudo real closed equivalent is an easy consequence of \ref{globdensPRC}. The rest of the arguments can be copied \emph{mutatis mutandis} replacing \(\ACVFG\) by the theory of algebraically closed fields and \(\pCFG\) by the theory of real closed fields.
